\documentclass[10pt,a4paper]{amsart}
\usepackage[utf8]{inputenc}
\usepackage[english]{babel}
\usepackage{amsmath}
\usepackage{amsfonts}
\usepackage{amssymb}
\usepackage{amsthm}
\usepackage[foot]{amsaddr}
\usepackage[backend=biber,style=numeric]{biblatex}
\usepackage{csquotes}
\usepackage{xcolor}
\usepackage{colortbl}
\usepackage{enumitem}
\usepackage{graphicx}
\usepackage{longtable}
\usepackage{xpatch}

\xpatchcmd{\paragraph}{\normalfont}{{\normalfont\bfseries}}{}{}
\xpatchcmd{\subparagraph}{\normalfont}{{\normalfont\itshape}}{}{}

\newtheorem{theorem}{Theorem}[section]
\newtheorem{prop}[theorem]{Proposition}

\newtheorem{coro}{Corollary}[theorem]
\newtheorem{lem}[theorem]{Lemma}

\theoremstyle{definition}
\newtheorem{definition}[theorem]{Definition}

\theoremstyle{remark}

\numberwithin{equation}{section}

\newcommand{\reft}[1]{Theorem~\ref{#1}}
\newcommand{\refT}[1]{Table~\ref{#1}}
\newcommand{\refp}[1]{Proposition~\ref{#1}}

\newcommand{\refd}[1]{Definition~\ref{#1}}

\newcommand{\refl}[1]{Lemma~\ref{#1}}

\newcommand{\boldify}[1]{\boldmath{#1}\unboldmath}

\newlist{ByCase}{enumerate}{3}
\setlist[ByCase,1]{label=\textbf{Case \arabic{*}}}
\setlist[ByCase,2]{label=\textbf{Subcase \arabic{ByCasei}.\arabic*}}
\setlist[ByCase,3]{label=\textbf{Subsubcase \arabic{ByCasei}.\arabic{ByCaseii}.\arabic*}}

\newcommand\N{\mathbb{N}}
\newcommand\Z{\mathbb{Z}}
\newcommand\Q{\mathbb{Q}}
\newcommand\R{\mathbb{R}}
\newcommand\F{\mathcal{F}}
\newcommand\G{\tilde{\F}}

\newcommand{\ab}[1]{{#1}_{(a,b)}}
\newcommand{\tab}[1]{{\overline{#1}}_{(a,b)}}
\newcommand\s{\ab s}
\newcommand\su{\ab u}
\newcommand\sv{\ab v}

\newcommand\tu{\tab u}
\newcommand\tv{\tab v}
\newcommand\tw{\tab w}
\newcommand\Fab{\ab \F}
\newcommand\dab{\ab d}
\newcommand\phiab{\ab \phi}

\newcommand\si[1]{{\s}_{#1}}
\newcommand\sui[1]{{\su}_{#1}}
\newcommand\svi[1]{{\sv}_{#1}}

\newcommand\tui[1]{{{\tu}}_{#1}}
\newcommand\tvi[1]{{{\tv}}_{#1}}
\newcommand\twi[1]{{{\tw}}_{#1}}

\author{Laurent Fallot}
\address{EPOC - UMR CNRS 5805 - Equipe PROMESS, Université de Bordeaux - Bordeaux-INP, Avenue des Facultés, CS60099, 33400 Talence, France}
\email{Laurent.Fallot@bordeaux-inp.fr}

\title{The monoid of numbers of the form $1<a^q/b^p<a$.}
\date{\today}

\begin{document}
\begin{abstract}
This paper is a study of the set of rational numbers of the form $1<a^q/b^p<a$ with $a$ and $b$ co-prime integers. 
The set $\Fab$ of these numbers, with an appropriate binary law, is a monoid isomorphic to $(\N,+,0)$. 
We identify the sequences of minimum and maximum record holders in $\Fab$ and prove that the first one converges to 1 while the second one converges to $a$.
We conclude that $\Fab$ is dense in the set of the real numbers comprise between 1 and a.
\end{abstract}

\maketitle

\section{Introduction}
Working on Collatz problem also known as the 3x+1 problem or Syracuse problem and many more, we initiate a study of numbers of the form $2^q/3^p$ where $p$ is any positive integer and $q$ is the smallest integer such that $2^q/3^p>1$.
This study gave some properties that are generalizable to numbers of the form $a^q/b^p$ where $1<a<b$ are co-prime integers. 
We introduce in this paper, the generalization of these results.

We propose a study of numbers of the form $\frac{a^q}{b^p}$ where $1<a<b$ are co-prime integers, $p\in\N$ and $q\in\N$ is such that $1\leq\frac{a^q}{b^p}<a$. Let $\Fab$ denotes the set of these numbers.

$\Fab$ is clearly countable. 
Moreover, the function $\phiab$ that naturally maps $\N$ to $\Fab$ permits to define an unconventional binary law on $\Fab$ making of $\Fab$ a monoid inferring that $\phiab: \N \rightarrow \Fab$ is a monoid isomorphism. 

Later, we define two sequences of $\Fab$, one strongly depending on each other. 
Each of these sequences enumerates the record holders of $\Fab$ of one type: minimum or maximum. 
Despite an apparent irregularity of $\Fab$, the definition of these sequences is global and does not depend on local properties. 
One of this sequence converges to 1 while the other one converges to $a$.

By the use of the sequence containing the minimum record holder of $\Fab$, we conclude that $\Fab$ is dense in $[1,a]_\R$. 
Moreover, the union of the $2^k\Fab$ on $k\in\Z$ admits $\R^+$ as its closure.

\section{Context}
The Collatz problem is based on a function whose definition is understandable to everyone.
This problem was initiate during the 30's and is not solved till today.
Many people, researchers or not, from different backgrounds have tried to reformulate the problem in their own jargon, but none has yet solved this conjecture.

We do not intend to solve it in this paper. 
We just use it as the context of this work.
A starting point of this function may be Wikipedia \cite{Wikipedia2023}.
For more references on this problem, we cite the overview of Lagarias \cite{Lagarias2021} which is regularly updated.

There are several definitions of this function. 
Each of these definitions gives a problem similar to the initial one.

For the purpose of this presentation of this work context, we use the compressed form of the function.
It may be defined as $T: \N\rightarrow\N$ with $T(n) = n/2$ if $n$ is even, else $T(n)=(3n+1)/2$.
Collatz problem claims that for any $n\in\N-\{0\}$ there exist $k\in\N$ such that $T^k(n) = 1$.

Terras \cite{Terras_1976} presents some results on the weak Collatz conjecture which consists in proving that for any $n\in\N$, there is a $k\in\N$ such that $T^k(n)<n$.
He particularly defines the stopping time of $n$ as the lowest $k$ such that $T^k(n)<n$.

Silva\cite{Silva1999} for his part, claims that for any $n<2^k$ and any $a\in\N$, $T^k(2^{k}a + n) = 3^p a + T^k(n)$ where $p$ is the number of times the odd part of $T$ is applied.
From this, without going into details, it is possible to prove that if $n$ has a stopping time of $k$, then for any $a\in\N$, $2^k a + n$ has a stopping time lower than or equal to $k$. Silva's tests on the stopping time of $2^k a + n$ do not show any $2^k a + n$ having a stopping time strictly lower than $k$. All of them accept $k$ as a stopping time as Silva claims in Proposition 6 of its paper\cite{Silva1999}.

Today, we do not have any proof that  $2^{k}a + n$ has exactly $k$ as its stopping time.
Trying to build its proof by contradiction leads to prove that $T^k(n)<n$ as soon as $2^k/3^p > 1$ which also seems to be verified by tests. 
An attempt to prove this new proposition by the use of the remainder representation introduced by Terras \cite{Terras_1976} quickly leads to have to prove that $2^k/3^p+1/3^p>2^{k-p}/ 3^p + 1$.

From these problems to be solved born an interest into getting a better knowledge of the numbers of the form $1<2^q/3^p<2$.
The study of these numbers gives some interesting properties.
None of these properties use any condition but $1<2<3$ and 2 and 3 are co-primes.

Below, we propose a generalization of these properties to numbers of the form $1<a^q/b^p<a$ where $a$ and $b$ are two co-prime integers satisfying $1<a<b$.

\section{Conventions and quick properties}
\label{sec:definitions}
Before we start our study, let us agree on some notations and definitions.
Then, we introduce some quick properties easily obtained from definitions.

The writing \boldify {$\lceil x \rceil$} denotes the ceiling function applied to $x\in\R$.
That is $\lceil x \rceil$ is the smallest integer greater than or equal to $x$. 
Thus, $\lceil x \rceil$ is the only integer $n$ verifying $x \leq n < x + 1$.

By the same way, \boldify {$\lfloor x \rfloor$} denotes the floor function applied to $x\in\R$.
That is $\lfloor x \rfloor$ is the greatest integer lower than or equal to $x$. 
So, $\lfloor x \rfloor$ is the only integer $n$ verifying $x - 1 < n \leq x$.

Let $a$ and $b$ be two co-prime positive integers such that $1<a<b$.
We are on the way to study rational numbers of the form $\frac{a^q}{b^p}$ where $p$ is any positive integer and $q$ is the smallest positive integer such that $1\leq \frac{a^q}{b^p}$. 
That is $q = \lceil p\log(b)/\log(a)\rceil$ where $\log$ denotes the natural logarithm function.
The set of these rational numbers will be denoted by \boldify {$\Fab$}.

The definition of such rational numbers infers that given $a$, $b$ and $p$, $q$ is unique.
More, to obtain $1\leq \frac{a^q}{b^p}$ we need that $p\leq q$ because $a<b$.
In summary, $q-p$ is a positive or null integer depending only on $a$, $b$ and $p$.
Let us use the notation \boldify {$\dab(p)$} for this difference. 
Thus, for any $f\in\Fab$, there exist a unique $p$ such that $f = \frac{a^{p+\dab(p)}}{b^p}$ by the definition of $\Fab$ inferring that $\dab(p)$ is also unique.

Finally, we define the function \boldify {$\phiab$} as follow.
\begin{equation*}
\begin{array}{lccl}
\phiab\colon & \N & \rightarrow & \Fab\\
& p & \to & \phiab(p) = \frac{a^{p+\dab(p)}}{b^p}
\end{array}
\end{equation*}

Let us claim now some quick properties of $\Fab$ and $\phiab$.
\begin{prop}
\label{prop:1_the_only_integer}
Given two co-prime positive integers $1<a<b$, the only integer that $\Fab$ contains is $1$.
\end{prop}

\begin{proof}
Note first that $\phiab(0) = \frac{a^{0+\dab(0)}}{b^0} = a^{\dab(0)}$. To have $1\leq \phiab(0) < a$, we need that $\dab(0) = 0$. Thus, $\phiab(0) = 1$ and $1\in \Fab$.

Now, for any rational number $r$ of the form $\frac{m}{n}$, the assertion $r\in\N$ infers that $m$ is a multiple of $n$.
Take any $f\in\Fab$. The number $f$ is necessarily of the form $\frac{a^{p+\dab(p)}}{b^p}$ by the definition of $\Fab$. 
As $a$ and $b$ are co-prime, the only possible case such that $b^p$ divides $a^{p+\dab(p)}$ is $b^p=1$. 
In such a case, $p=0$ and $\dab(p)=0$. 
As a consequence, $f = \phiab(0) = 1$.
So, $\Fab$ cannot contain any positive integer other than 1.
\end{proof}

Let us now state that any member of $\Fab$ has a unique representation of the form $\frac{a^{p+\dab(p)}}{b^p}$.

\begin{prop}
\label{prop:unique_representation}
Given two co-prime integers $a$ and $b$ verifying that $1<a<b$ and two positive integers $p_1$ and $p_2$, $\phiab(p_1)=\phiab(p_2)$ if and only if $p_1=p_2$.
\end{prop}

\begin{proof}
On the one hand, the fact that $p_1=p_2$ is sufficient to immediately obtain $\phiab(p_1)=\phiab(p_2)$ from the definitions of $\phiab$ and $\dab$. 
Indeed, assume that $p_1=p_2$. The uniqueness of $\dab(p)$ for a given $p$ implies that $\dab(p_1)=\dab(p_2)$. 
Thus, $\frac{a^{p_2+\dab(p_2)}}{b^{p_2}}=\frac{a^{p_1+\dab(p_1)}}{b^{p_1}}$ and $\phiab(p_1)=\phiab(p_2)$.

On the other hand, to justify that $p_1=p_2$ is necessary for $\phiab(p_1)=\phiab(p_2)$, let us reason by contradiction. 

Suppose that given two positive integers $p_1$ and $p_2$ such that $p_1\neq p_2$ we have $\phiab(p_1)=\phiab(p_2)$. 
We can assume that $p_2<p_1$ without any loss of generalities.
By hypothesis, we have $\phiab(p_1)=\phiab(p_2)$.
This reads $\frac{a^{p_1+\dab(p_1)}}{b^{p_1}}=\frac{a^{p_2+\dab(p_2)}}{b^{p_2}}$.
Thus $\frac{a^{p_1+\dab(p_1)}}{b^{p_1}}/\frac{a^{p_2+\dab(p_2)}}{b^{p_2}} = 1$.
A factorization of this equation gives $\frac{a^{p_1-p_2+\dab(p_1)-\dab(p_2)}}{b^{p_1-p_2}} = 1$.
Because $a$ and $b$ are co-prime, we need to have $p_1-p_2+\dab(p_1)-\dab(p_2) = 0$ and $p_1-p_2=0$.
So $p_1=p_2$ which contradicts the hypothesis that they are different.
In conclusion, to get $\frac{a^{p_1+\dab(p_1)}}{b^{p_1}}=\frac{a^{p_2+\dab(p_2)}}{b^{p_2}}$ we must have $p_1=p_2$ and \refp{prop:unique_representation} holds.
\end{proof}

\section{The function $\phiab$ is an isomorphism of monoids.}
\label{sec:phi_isomorphism}
In this section, we first establish that $\phiab$ is a bijection. 
Then, we define an operation on $\Fab$ induced by $\phiab$ from the addition on $\N$. 
After a more precise description of this operation, we conclude that $\Fab$ is a commutative monoid.
Finally, we state that $\phiab: \N\rightarrow\Fab$ is a monoid isomorphism.

We claim the following quick theorem.
\begin{theorem}
\label{thm:F_is_countable}
Whatever the two co-prime integers $1<a<b$ are, the set $\Fab=\left\lbrace \phiab(p) = \frac{a^{p+\dab(p)}}{b^p},~ p\in\N\right\rbrace$ where $\dab(p)$ is the smallest integer such that $1<\frac{a^{p+\dab(p)}}{b^p}$ is a countable set.
\end{theorem}

\begin{proof}
The set $\Fab$ is countable if and only if there exist at least one bijection between $\N$ and it by the definition of the property "is countable". Let us prove that $\phiab$ defined in section \ref{sec:definitions} is a bijection.

Firstly, the domain of $\phiab$ is $\N$. 
Effectively, for any finite $p\in\N$, $b^p$ is different from 0 by the hypothesis $b$ is positive. 
Then, as the exponential function $a^x$ tends to $+\infty$ when $x\to+\infty$, there exist an infinity of $q\in\N$ such that $a^q > b^p$. Let us take the smallest $q$ so that this proposition is verified. As we supposed that $a<b$, $a^q > b^p$ infers that $q>p$ and we can write that $q = p + \dab(p)$. Thus, $\frac{a^{p+\dab(p)}}{b^p}$ exist and is a member of $\Fab$ by the construction of this set. This is summarized by the domain of $\phiab$ is $\N$.

Secondly, as $\Fab = \phiab(\N)$ by its definition, so any $f\in\Fab$ has an antecedent in $\N$.

Lastly, \refp{prop:unique_representation} claims that any $f\in\Fab$ cannot have more than one antecedent in $\N$.

In conclusion, $\phiab$ is a bijection and $\Fab$ is countable.
\end{proof}

Now, define an operation on $\Fab$ induced by $\phiab$ from the addition on $\N$.

\begin{definition}
\label{def:star}
Given any $1<a<b$, two co-prime positive integers, the symbol "$\star$" denotes the binary operation from $\Fab\times \Fab$ in $\Fab$ defined by $$\phiab(p_1)\star\phiab(p_2) = \phiab(p_1+p_2).$$
\end{definition}

In this definition we use the symbol "$\star$" to denote this operation to avoid possible confusions with "$\cdot$", the multiplication on $\R$ which is also used later.

Now, we can claim the following theorem.

\begin{theorem}
\label{thm:star_is_a_monoid_law}
For any $a$ and $b$ co-prime positive integers satisfying $1<a<b$, $\phiab\colon (N,+, 0) \rightarrow (\Fab,\star, 1)$ is a monoid isomorphism.
\end{theorem}

\begin{proof}
This theorem may be easily deduced from the definition of $\star$, from \refp{thm:F_is_countable} and from \refp{prop:1_the_only_integer}.

Take any pair $(f_1,f_2)\in\Fab\times\Fab$. 
As $\phiab$ is a bijection by \reft{thm:F_is_countable}, there exist $(p_1,p_2) \in \N\times\N$ such that $f_1=\phiab(p_1)$ and $f_2=\phiab(p_2)$. 
Now, $p_1+p_2\in\N$ so that $\phiab(p_1+p_2)\in\Fab$ by the definition of $\Fab$. 
Thus $\star$ is an intern law.

The associativity of $\star$ may be easily obtained from the associativity of $+$ on $\N$ and \refp{prop:unique_representation}. 
Indeed, take any $f_1,~f_2$ and $f_3$ three members of $\Fab$.
By \reft{thm:F_is_countable}, for each of the $f_i,~i=1,2,3$ there exist a positive integer $p_i$ such that $f_i = \phiab(p_i)$. 
The definition of $\star$ induces that 
\begin{equation*}
\begin{array}{llll}
\phiab(p_1)\star(\phiab(p_2)\star\phiab(p_3))&=&\phiab(p_1) \star\phiab(p_2+p_3)&\\
&=&\phiab(p_1+(p_2+p_3))&\\
&=&\phiab((p_1+p_2)+p_3).&
\end{array}
\end{equation*}
Finally, applying two times the definition of $\star$ we obtain $$\phiab(p_1)\star(\phiab(p_2)\star\phiab(p_3))=(\phiab(p_1)\star\phiab(p_2))\star\phiab(p_3)$$ proving by the same way that $\star$ is associative.

The operation $\star$ is commutative. 
Effectively, Take $f_1=\phiab(p_1)$ and $f_2=\phiab(p_2)$ two members of $\Fab$.
By definition of $\star$, $f_1\star f_2 = \phiab(p_1+p_2)$. 
The commutativity of + in $\N$ gives $\phiab(p_1+p_2)=\phiab(p_2+p_1)$. 
That is $f_1\star f_2 = f_2\star f_1$.

Now, by \refp{prop:1_the_only_integer}, for any $f=\phiab(p) \in\Fab$, we can write that $f\star 1 = \phiab(p)\star\phiab(0) = \phiab(p+0)=  \phiab(p)$. 
By the commutativity of $\star$ we obtain $f\star 1 = 1\star f = f$. 
So 1 is the neutral element of $\Fab$.

To conclude, $(\Fab,\star)$ is a commutative monoid. 
Moreover, the definition of $\star$ claims that $\phiab(p_1+p_2) = \phiab(p_1)\star\phiab(p_2)$ and by \refp{prop:1_the_only_integer} $\phiab(0) = 1$. 
So $\phiab$ is a morphism of monoids and is even an isomorphism of monoids by \reft{thm:F_is_countable}.
\end{proof}

Now that we state that $(\Fab,\star)$ is a commutative monoid, let us have a closer look at how the operation $\star$ acts on the fractions of the form $\frac{a^{p+\dab(p)}}{b^p}$.

\begin{theorem}
\label{thm:star_is_merely_the_product}
Given $1<a<b$ two co-prime positive integers. 
Let $f_1$ and $f_2$ be two members of $\Fab$. 
Then $$f_1 \star f_2 = \left\lbrace\begin{array}{ll}
f_1\cdot f_2& \mbox{ if } f_1\cdot f_2 < a\\
\frac{1}{a}\cdot f_1\cdot f_2 & \mbox{ if } a\leq f_1\cdot f_2
\end{array}\right.$$
where "$\cdot$" denotes the multiplication of two real numbers.
\end{theorem}

\begin{proof}
We know that $f_1 \star f_2 \in\Fab$ by \reft{thm:star_is_a_monoid_law}.
Define $p_1$ and $p_2$ such that $f_1=\phiab(p_1)$ and $f_2=\phiab(p_2)$.
By the definition of $\phiab$ we have $f_1 = \frac{a^{p_1+\dab(p_1)}}{b^{p_1}}$ and $f_2=\frac{a^{p_2+\dab(p_2)}}{b^{p_2}}$.
By the way, $f_1 \star f_2 = \phiab(p_1+p_2) = \frac{a^{p_1+p_2 + \dab(p_1+p_2)}}{b^{p_1+p_2}}$.
Moreover, $1\leq f_1\cdot f_2 < a^2$.

Consider separately the two cases introduced by \reft{thm:star_is_merely_the_product}.
\begin{ByCase}
\item Suppose that $1\leq f_1\cdot f_2 < a$. 
We have, $f_1\cdot f_2 = \frac{a^{p_1+\dab(p_1)}}{b^{p_1}} \cdot \frac{a^{p_2+\dab(p_2)}}{b^{p_2}}$.
This reads $1\leq f_1\cdot f_2 = \frac{a^{p_1+p_2+\dab(p_1)+\dab(p_2)}}{b^{p_1+p_2}}<a$.
So $f_1\cdot f_2 \in \Fab$.
This infers that $\phiab(p_1+p_2) = f_1\cdot f_2$.
More, $\dab(p_1+p_2) = \dab(p_1) + \dab(p_2)$ by identification. 
To conclude that the first case of \reft{thm:star_is_merely_the_product} holds.
\item Suppose now that $a\leq f_1\cdot f_2 < a^2$.
Clearly, $1\leq \frac{1}{a} \cdot f_1\cdot f_2 < a$.
Note that $\frac{1}{a} \cdot f_1\cdot f_2 = \frac{a^{p_1+p_2+\dab(p_1)+\dab(p_2)-1}}{b^{p_1+p_2}}$.
Thus, $\frac{1}{a} \cdot f_1\cdot f_2 = \phiab(p_1+p_2)\in\Fab$.
Furthermore, we get $\dab(p_1+p_2) = \dab(p_1) + \dab(p_2) - 1$ by identification.
Thus, the second case of \reft{thm:star_is_merely_the_product} also holds.
\end{ByCase}
\end{proof}

To conclude with this section, we can extend $\phiab$ on $\Z$ getting by the same way a group $(\G_{a,b}, \star)$ extending $(\Fab,\star)$.
As it is out of interest in this paper, we just give some indications on how to proceed without going into details.

Define $\tilde{\phi}_{a,b}(p)$ the extension of $\phiab$ on $\Z$ by $\tilde{\phi}_{a,b}(p)=\phiab(p)$ and $\tilde{\phi}_{a,b}(-p)=\frac{1}{\phiab(p)}$ where $p>0$. 
Note that $1\leq \tilde{\phi}_{a,b}(p) < a$ and $\frac{1}{a}\leq\tilde{\phi}_{a,b}(-p)\leq1$ when $0\leq p$.
Now, define $\G_{a,b} = \tilde{\phi}_{a,b}(\Z)$.
Finally, extends $\star$ on $\G_{a,b}$ by the use of the same definition as for $\Fab$.
Clearly, $(\G_{a,b},\star)$ is a commutative monoid as $(\Fab,\star)$ is.
Moreover, for every $p\in\Z$, $\tilde{\phi}_{a,b}(p)\star\tilde{\phi}_{a,b}(-p) = 1$.
So $(\G_{a,b},\star)$ is a group and $\tilde{\phi}_{a,b}$ is an isomorphism of groups between $(\Z,+,0)$ and $(\G_{a,b},\star,1)$.

\section{Record holders}
\label{sec:record_holders}
To establish the proof of the main theorem of this paper, we proceed as follow.
Firstly, we build a sequence of $\Fab$ that converges to 1.
In the same time, we get another convergent sequence which limit is $a$.
Secondly, we prove that any $x\in[1,a]_\R$ is the limit of a convergent sequence.
Thirdly, we deduce that $\Fab$ is dense in $[1,a]_\R$.
Finally, we conclude that our main theorem holds.

For our purpose, we need to build a sequence $\su$ of $\Fab$ that converges to 1.
The aim of this section is to define such a sequence.

Below, by minimum record holder understand any $\phiab(p)$ such that $\phiab(p) = \inf\{\phiab(p'),~1\leq p'\leq p\}$.
Reciprocally, a maximum record holder is any $\phiab(p)$ such that $\phiab(p) = \sup\{\phiab(p'),~1\leq p'\leq p\}$.

We observe record holders numerically obtained for $a=2$ and $b=3$.
We highlight how they may be obtained and how the minimum record holders and maximum ones interact.
We define the corresponding sequences and prove that they contain all the record holders and only them.

By the way, we obtain that $\su$ is composed of all the minimum record holders of $\Fab$.
But, to obtain $\su$, we have to define a second one, $\sv$, containing all the maximum record holders of $\Fab$ and only them.
Both sequences are strongly related one to each other so that we have to define and manage them in the same time.

To get an idea of these two sequences, let us observe the record holders of the sequence $\left(\phi_{(2,3)}(p)\right)_{p\in\N}$.

\refT{tab:record_holders} summarizes the record holders of the form $\phi_{(2,3)}(p)$ for $1\leq p \leq 2^{15}$. 
These results were obtained with a specific script written for Maxima\cite{Joyner2006}\cite{Maxima2023}, a Computer Algebra System under GPL license.
The used algorithm consists in a simple for loop on $p$ which computes $\dab(p)$ and $\phiab(p)$.
Each time it founds a new record holder, the script put this number and its related informations at the end of an initially empty list.

\begin{table}[ht]
\tiny
\caption{Repartition of the record holders $\leq 2^{15}$ for $a=2$ and $b=3$}
\label{tab:record_holders}

\begin{tabular}{ccccc|ccccc}

\multicolumn{5}{c|}{Minimum record holders}& \multicolumn{5}{c}{Maximum record holders} \\
$\phi(p)$ & $p$ & $d(p)$ & $\Delta p$ & $\Delta d(p)$ & $\phi(p)$ & $p$ & $d(p)$ & $\Delta p$ & $\Delta d(p)$ \\
\hline
$\frac{2^2}{3}$ & 1 & 1 & 1 & 1 & $\frac{2^2}{3}$ & 1 & 1 & 1 & 1 \\  [0.1em]
         &  &  &  &  & $\frac{2^4}{3^2}$ & 2 & 2 & 1 & 1 \\  [0.1em]
        $\frac{2^5}{3^3}$ & 3 & 2 & 2 & 1 &  &  &  &  &  \\  [0.1em]
        $\frac{2^8}{3^5}$ & 5 & 3 & 2 & 1 &  &  &  &  &  \\  [0.1em]
         &  &  &  &  & $\frac{2^{12}}{3^7}$ & 7 & 5 & 5 & 3 \\  [0.1em]
         &  &  &  &  & $\frac{2^{20}}{3^{12}}$ & 12 & 8 & 5 & 3 \\  [0.1em]
        $\frac{2^{27}}{3^{17}}$ & 17 & 10 & 12 & 7 &  &  &  &  &  \\  [0.1em]
        $\frac{2^{46}}{3^{29}}$ & 29 & 17 & 12 & 7 &  &  &  &  &  \\  [0.1em]
        $\frac{2^{65}}{3^{41}}$ & 41 & 24 & 12 & 7 &  &  &  &  &  \\  [0.1em]
         &  &  &  &  & $\frac{2^{85}}{3^{53}}$ & 53 & 32 & 41 & 24 \\  [0.1em]
        $\frac{2^{149}}{3^{94}}$ & 94 & 55 & 53 & 31 &  &  &  &  &  \\  [0.1em]
        $\frac{2^{233}}{3^{147}}$ & 147 & 86 & 53 & 31 &  &  &  &  &  \\  [0.1em]
        $\frac{2^{317}}{3^{200}}$ & 200 & 117 & 53 & 31 &  &  &  &  &  \\  [0.1em]
        $\frac{2^{401}}{3^{253}}$ & 253 & 148 & 53 & 31 &  &  &  &  &  \\  [0.1em]
        $\frac{2^{485}}{3^{306}}$ & 306 & 179 & 53 & 31 &  &  &  &  &  \\  [0.1em]
         &  &  &  &  & $\frac{2^{570}}{3^{359}}$ & 359 & 211 & 306 & 179 \\  [0.1em]
         &  &  &  &  & $\frac{2^{1055}}{3^{665}}$ & 665 & 390 & 306 & 179 \\  [0.1em]
        $\frac{2^{1539}}{3^{971}}$ & 971 & 568 & 665 & 389 &  &  &  &  &  \\  [0.1em]
        $\frac{2^{2593}}{3^{1636}}$ & 1636 & 957 & 665 & 389 &  &  &  &  &  \\  [0.1em]
        $\frac{2^{3647}}{3^{2301}}$ & 2301 & 1346 & 665 & 389 &  &  &  &  &  \\  [0.1em]
        $\frac{2^{4701}}{3^{2966}}$ & 2966 & 1735 & 665 & 389 &  &  &  &  &  \\  [0.1em]
        $\frac{2^{5755}}{3^{3631}}$ & 3631 & 2124 & 665 & 389 &  &  &  &  &  \\  [0.1em]
        $\frac{2^{6809}}{3^{4296}}$ & 4296 & 2513 & 665 & 389 &  &  &  &  &  \\  [0.1em]
        $\frac{2^{7863}}{3^{4961}}$ & 4961 & 2902 & 665 & 389 &  &  &  &  &  \\  [0.1em]
        $\frac{2^{8917}}{3^{5626}}$ & 5626 & 3291 & 665 & 389 &  &  &  &  &  \\  [0.1em]
        $\frac{2^{9971}}{3^{6291}}$ & 6291 & 3680 & 665 & 389 &  &  &  &  &  \\  [0.1em]
        $\frac{2^{11025}}{3^{6956}}$ & 6956 & 4069 & 665 & 389 &  &  &  &  &  \\  [0.1em]
        $\frac{2^{12079}}{3^{7621}}$ & 7621 & 4458 & 665 & 389 &  &  &  &  &  \\  [0.1em]
        $\frac{2^{13133}}{3^{8286}}$ & 8286 & 4847 & 665 & 389 &  &  &  &  &  \\  [0.1em]
        $\frac{2^{14187}}{3^{8951}}$ & 8951 & 5236 & 665 & 389 &  &  &  &  &  \\  [0.1em]
        $\frac{2^{15241}}{3^{9616}}$ & 9616 & 5625 & 665 & 389 &  &  &  &  &  \\  [0.1em]
        $\frac{2^{16295}}{3^{10281}}$ & 10281 & 6014 & 665 & 389 &  &  &  &  &  \\  [0.1em]
        $\frac{2^{17349}}{3^{10946}}$ & 10946 & 6403 & 665 & 389 &  &  &  &  &  \\  [0.1em]
        $\frac{2^{18403}}{3^{11611}}$ & 11611 & 6792 & 665 & 389 &  &  &  &  &  \\  [0.1em]
        $\frac{2^{19457}}{3^{12276}}$ & 12276 & 7181 & 665 & 389 &  &  &  &  &  \\  [0.1em]
        $\frac{2^{20511}}{3^{12941}}$ & 12941 & 7570 & 665 & 389 &  &  &  &  &  \\  [0.1em]
        $\frac{2^{21565}}{3^{13606}}$ & 13606 & 7959 & 665 & 389 &  &  &  &  &  \\  [0.1em]
        $\frac{2^{22619}}{3^{14271}}$ & 14271 & 8348 & 665 & 389 &  &  &  &  &  \\  [0.1em]
        $\frac{2^{23673}}{3^{14936}}$ & 14936 & 8737 & 665 & 389 &  &  &  &  &  \\  [0.1em]
        $\frac{2^{24727}}{3^{15601}}$ & 15601 & 9126 & 665 & 389 &  &  &  &  &  \\  [0.1em]
         &  &  &  &  & $\frac{2^{25782}}{3^{16266}}$ & 16266 & 9516 & 15601 & 9126 \\  [0.1em]
         &  &  &  &  & $\frac{2^{50509}}{3^{31867}}$ & 31867 & 18642 & 15601 & 9126 \\
\end{tabular}
\end{table}
    
To highlight the interaction between the minimum record holders and the maximum ones, \refT{tab:record_holders} is divided vertically into two parts. 
The left part lists the minimum record holders while the right one describes the maximum record holders.
For the same purpose, the records are sorted by increasing p-value.

Each row of the table corresponds to a new record holder.
If it is a minimum one, the right part is left empty.
Reciprocally, for a maximum record holder, the left part is left blank.
An exception to this rule is the first row. 
It describes $\phi_{(2,3)}(1)$ the first member of $\F_{(2,3)}$ that is considered by our script.
As it is the first explored number, we assume that it is the first minimum record holder and the first maximum one at the same time.
So we list it as so.

For each of the two vertical parts, the first column gives the record holder $\phi_{(2,3)}(p)$. 
The second column contains the value of $p$ for this record holder.
The third column collects the computed value of $d_{(2,3)}(p)$.
The fourth column named $\Delta p$ is the difference between the value of $p$ of the current record holder and the value of $p$ of the previous record holder of the same type.
The fifth and last column named $\Delta d(p)$ details the difference between $d_{(2,3)}(p)$ of the current record holder and the same value for the previous record holder of the same type.

Observe the evolution of the maximum record holders.
Note that we do not consider $\phi_{(2,3)}(0)$.
We avoid it because it is in fact $\inf \{f\in\Fab\}$ and it is the only integer contained in this set.
Our analysis starts with $\phi_{(2,3)}(1)$.

The first one is $\frac{2^2}{3^1}$ as previously announced as it is the first considered $\phi_{(2,3)}(p)$.
The second one is $\frac{2^4}{3^2}$ is a maximum record holder which may be seen as the product of the last maximum record holder found by the last minimum one.
Look at the following maximum record holders.
We can see that this behaviour repeats each time.

Observe now the evolution of the minimum record holders.
By the same way, we can see that we can obtain the next minimum record holder by multiplying the last minimum we found by last maximum record holder found divided by $2$.

Note that this sequence is divided into phases where one of the two record holders evolves while the other one stays invariant. 
Let us try to understand when this happens.

Take a phase where the maximum record holder evolves. 
For instance, we have the subsequence which start with $\frac{2^{8}}{3^5}$.
Note that $\frac{2^{8}}{3^5}$ is the last minimum record holder that we found until now while $\frac{2^4}{3^2}$ is the last maximum one.
Now $\frac{2^4}{3^2}\cdot \frac{2^{8}}{3^5}= \frac{2^{12}}{3^7}\approx 1.87288523< 2$ and this number is the next maximum record holder.
Furthermore, $\frac{2^{12}}{3^7}\cdot\frac{2^8}{3^5} = \frac{2^{20}}{3^{12}}\approx 1.97308074 < 2$ is the new maximum record holder.
But $\frac{2^{20}}{3^{12}} \cdot \frac{2^8}{3^5} = \frac{2^{28}}{3^{17}}\approx 2.07863650 > 2$ and the phase of maximum record holder growth stops.
A new phase starts in which minimum record holder decreases.
It starts at $\frac{2^{20}}{3^{12}}$ and the next minimum record holder is $\frac{2^{28}}{3^{17}}/2 = \frac{2^{27}}{3^{17}}\approx 1.03931825$.
The last maximum record holder found is $\frac{2^{20}}{3^{12}}$.
Pursue until the end of this phase and observe that each time the product of the last minimum record holder found by 
$\frac{2^{20}}{3^{12}}$ gives a result greater than $a$.
Once $\frac{2^{65}}{3^{41}}$ obtained as the last minimum record holder, we have $\frac{2^{65}}{3^{41}}\cdot \frac{2^{20}}{3^{12}} = \frac{2^{85}}{3^{53}} \approx 1.99582809 < 2 $.
Then starts a new phase where the maximum record holder increases.

Go further in \refT{tab:record_holders} to observe that this pattern repeats in the whole table.
So, while the product of the last maximum record holder found by the minimum one is lower than 2, we obtain a new maximum record holder in \refT{tab:record_holders}, otherwise we get a new minimum record holder.

Generalize and expand this behaviour with $1<a<b$, two co-prime integers, to define two asynchronous sequences of $\Fab$. 
For the sack of simplicity, we mimic \refT{tab:record_holders} by defining a sequence $\s$ whose elements are pairs $\si i = (\sui i,\svi i)$ where $\su$ is the sequence that contains the minima we found while $\sv$ contains the maxima.

\begin{definition}
\label{def:define_s}
Let $a$ and $b$ be two co-prime integers such that $1<a<b$. 
We recursively define a sequence  $\s=((\su,\sv))$ of pairs of $\R\times\R$  by $\si 0 = (\sui 0,\svi 0) = \left(\phiab(1),\phiab(1)\right)$ and 
$$
(\sui {i+1},\svi {i+1}) = \left\lbrace\begin{array}{ll}
(\sui i, \sui i \svi i) & \mbox{ when } \sui i \svi i<a\\
(\frac{1}{a} \sui i \svi i,\svi i) & \mbox{ when } a \leq \sui i \svi i.\\
\end{array}\right.
$$
\end{definition}

Our next goal is to prove that the distinct $\sui i$ and the distinct $\svi i$ are record holders of $\Fab$ and all the record holders of $\Fab$ are present either in $\su$ or $\sv$ depending on their type.

\begin{theorem}
\label{thm:record_holders}
Take any pair of co-prime integers $1<a<b$. 
Let $\su$ and $\sv$ be the two sequences defined in \refd{def:define_s}. 
The set $\lbrace \sui i, i\in\N\rbrace$ is the set of all the minimum record holders of $\Fab$ and the set $\lbrace \svi i, i\in\N\rbrace$ is the set of all the maximum record holders of $\Fab$
\end{theorem}

Decompose the proof of \reft{thm:record_holders} into lemmas.
The first one states that any member of $\s$ is in $\Fab\times\Fab$.

\begin{lem}
\label{lem:s_in_FxF}
The sequence $\s$ is a sequence of $\Fab\times\Fab$.
\end{lem}

\begin{proof}
Let us prove \refl{lem:s_in_FxF} by induction.
\paragraph{Base case} Clearly, $\phiab(1)\in\Fab$ by the definition of this set. 
As an immediate consequence, $\si 0=(\phiab(1).\phiab(1))\in\Fab\times\Fab$.
\paragraph{Induction case} Suppose that $\si i=(\sui i,\svi i) \in \Fab\times\Fab$ for a given $i$ and prove that $\si {i+1}=(\sui {i+1},\svi {i+1}) \in \Fab\times\Fab$. 

As $\sui i$ is supposed to be a member of $\Fab$, there exist $p_1$ such that $\sui i = \phiab(p_1)$ and $1<\sui i<a$. 
For the same reasons, there exist $p_2$ such that $\svi i = \phiab(p_2)$ and $1<\svi i<a$.

Note that $1 < \sui i \svi i < a^2$ so we have two possibilities: $1<\sui i \svi i<a$ and $a\leq \sui i \svi i<a^2$. 
Let us consider these two possibilities separately.

\subparagraph{Case 1: $1<\sui i \svi i < a$.}

In this case, \refd{def:define_s} defines $\sui {i+1} = \sui i$ and $\svi {i+1} = \sui i \svi i$. 
The hypothesis of induction says that $\sui {i+1} = \sui i\in\Fab$. 

From another point of view, the hypothesis of this case is $1<\sui i \svi i < a$ inferring that $1<\phiab(p_1)\phiab(p_2)<a$. 
With the help of \reft{thm:star_is_merely_the_product} we get $\phiab(p_1)\phiab(p_2) = \phiab(p_1)\star\phiab(p_2) = \phiab(p_1+p_2)$. 
In conclusion, $\svi {i+1} = \phiab(p_1+p_2)\in\Fab$ by the construction of this set.

\subparagraph{Case 2: {$a\leq \sui i \svi i < a^2$.}}

In this case, \refd{def:define_s} defines $\sui {i+1} = \frac{1}{a} \sui i \svi i$ and $\svi {i+1} = \svi i$. 
It is immediate to see that $\svi {i+1} = \svi i\in\Fab$ from the hypothesis of induction. 

Concerning $\sui {i+1}$, it is supposed that $a\leq \sui i \svi i < a^2$. 
That leads to $1<\frac{1}{a}\phiab(p_1)\phiab(p_2)<a$. 
\reft{thm:star_is_merely_the_product} claims that $\frac{1}{a}\phiab(p_1)\phiab(p_2) = \phiab(p_1)\star\phiab(p_2) = \phiab(p_1+p_2)$ by \refd{def:star}. 
Thus, $\sui {i+1} = \phiab(p_1+p_2)\in\Fab$.

So, whatever the case is, $\si {i+1} = (\sui {i+1},\svi {i+1})$ is a member of $\Fab\times\Fab$. 

To conclude, by induction we proved that $\s$ is a sequence of $\Fab\times\Fab$.
\end{proof}

Before going further, note that for any $i\in\N$, if $\si i = (\phiab(p_1), \phiab(p_2))$ then, either  $\si {i+1} = (\phiab(p_1), \phiab(p_1)\star\phiab(p_2))$ when $\phiab(p_1)\cdot\phiab(p_1)<a$ or $\si {i+1} = (\phiab(p_1)\star\phiab(p_2),\phiab(p_2))$ when $a<\phiab(p_1)\cdot\phiab(p_1)$.
Thus the following lemma that may be considered as a corollary of \refl{lem:s_in_FxF}.

\begin{lem}
\label{lem:next_rh_is_ui_star_vi}
The sequence $\s$ defined in \refd{def:define_s} may be also defined as $\si 0 = (\phiab(1),\phiab(1))$ and
for any $i\in\N$,
$$(\sui {i+1},\svi {i+1}) = \left\lbrace\begin{array}{ll}
(\sui i, \sui i \star \svi i) &\mbox{ when } \sui i \svi i<a\\
(\sui i\star \svi i,\svi i) & \mbox{ when } a\leq \sui i \svi i.
\end{array}\right.$$
\end{lem}

Next lemma claims that neither $\sui i$ nor $\svi i$ may be infinitely constant.

\begin{lem}
\label{lem:u_v_both_change}
For all $i\in\N$, there exist an integer $k>0$ such that $\sui {i+k} \neq \sui i$ and $\svi {i+k} \neq \svi i$.
\end{lem}

\begin{proof}
Take any $i\in\N$ and observe the subsequence $(\si {i+k})_{k>0}$ extracted from $\s$. \refd{def:define_s} introduces two different cases. 
Let us consider them separately.

\subparagraph{Case 1: $\sui i \svi i < a$.}  
In such a case, \refd{def:define_s} declares that $\svi {i+1} = \sui i \svi i \neq \svi i$.
So, for any $k\geq 1$, $\svi {i+k} \neq \svi i$. 
Now, $\sui {i+1} = \sui i$.

Suppose that for any $k>0$, $\sui{i+k} = \sui i$. 
This implies that for any $k>0$, $\sui {i+k}\svi {i+k}<a$.
It is easy to see that, in such a case, for any $k>0$, $\svi {i+k} = \sui i^k \cdot \svi i$. 
As $\sui i\in\Fab$, $\sui i>1$ and the sequence $(\sui i^k \cdot \svi i)_{k\in\N} $ is a strictly increasing sequence which tends to $+\infty$.
As a consequence, there exist $k>0$ such that $\sui i^k \cdot \svi i > a$.
Take the lowest $k$ so that $\sui i^k \cdot \svi i > a$.
For this $k$, $\sui{i+k+1} = \frac{1}{a}\sui {i+k}\svi {i+k}$ by \refd{def:define_s}.
This contradicts the hypothesis assuming that for any $k>0$, $\sui {i+k} = \sui i$.
So this hypothesis does not hold and there exist a $k>0$ so that $\sui {i+k} \neq \sui i$. 
As $\svi {i+k} \neq \svi i$, we are done for this case.

\subparagraph{Case 2: $a \leq \sui i \svi i$.}
 In this case, \refd{def:define_s} sets the next term of $\su$ as $\sui {i+1} = \frac{1}{a} \sui i \svi i \neq \sui i$  and $\svi {i+1} = \svi i$.
From this, we can affirm that for any $k\geq 1$, $\sui {i+k} \neq \sui i$.

Let us suppose now that for any $k>0$, $a\leq \sui {i+k}\svi {i+k}$.
By a way similar to the one used in previous case, we can prove that $\sui {i+k} = \sui i \left(\frac{\svi i}{a}\right)^k$. 
As $\svi i\in\Fab$, $\svi i < a$ and $\frac{\svi i}{a}<1$.
So the sequence $\left(\sui i\left(\frac{\svi i}{a}\right)^k\right)_{k\in\N}$ converges to 0. 
This means that there exist $k>0$ so that $\sui i\left(\frac{\svi i}{a}\right)^k<1$ which contradicts the hypothesis assuming that such a $k$ does not exist.
We can conclude that there exist $k>0$ such that $\sui {i+k} \neq \sui i$ and $\svi {i+k} \neq \svi i$.

In conclusion of this reasoning, in both cases we obtain that there exist $k>0$ such that $\sui {i+k} \neq \sui i$ and $\svi {i+k} \neq \svi i$.

So under the conditions of \refl{lem:u_v_both_change} its conclusion holds.
\end{proof}

In summary of \refl{lem:u_v_both_change} and its proof, on the one hand, we can see that $\su$ monotonously decreases because when this sequence varies $\sui i$ is multiplied by $\frac{\svi i}{a}<1$.
Moreover, $\su$ never stays infinitely constant.
Furthermore, $\su$ admits 1 as a lower bound because all of its elements are members of $\Fab$.
Thus, $\su$ converges. 

On the other hand, $\sv$ monotonously increases. 
Indeed, when $\svi i$ changes, we multiply it by $\sui i>1$.
More, $\sv$ never stays infinitely constant.
Now, all the $\svi i\in\Fab$. 
Therefore, $a$ is a upper bound of $\sv$.
In conclusion, $\sv$ also converges.

Thus the following lemma that we use later to prove the next theorem.
\begin{lem}
\label{lem:u_and_v_converge}
As they are defined in \refd{def:define_s}, the sequence $\su$ converges to a limit $l_1\geq 1$ and the sequence $\sv$ converges to a limit $l_2 \leq a$.
\end{lem}

Moreover, the proof of \refl{lem:u_v_both_change} points out that if $\su$ remains constant between the indices $i$ and $i+k$ then $\svi {i+k} = \svi i \cdot \sui i ^k$.
Reciprocally, if $\sv$ does not change between the indices $i$ and $i+k$ then $\sui {i+k}=\sui i\left(\frac{\svi i}{a}\right)^k$. 
This is our next lemma.

\begin{lem}
\label{lem:u_i+k_and_v_i+k}
As $\su$ and $\sv$ are defined in \refd{def:define_s}, if $\su$ remains constant between the indices $i$ and $i+k$ then $\svi {i+k} = \svi i \cdot \sui i ^k$. Reciprocally, if $\sv$ does not change between the indices $i$ and $i+k$ then $\sui {i+k}=\sui i\left(\frac{\svi i}{a}\right)^k$
\end{lem}

As $\phiab$ is a bijection, its inverse denoted $\phiab^{-1}$ is also a bijection.
This function maps any $f\in\Fab$ to $\phiab^{-1}(f) = p\in\N$ such that $\phiab(p) = f$.

Now we can introduce the proof of \reft{thm:record_holders}.

\begin{proof}[Proof of \reft{thm:record_holders}]
Define first the sequence $\pi = (\pi_i)_{i\in\N}$ where for every $i\in\N$, $\pi_i = \sup\lbrace\phiab^{-1}(\sui i), \phiab^{-1}(\svi i)\rbrace$.
This sequence is a monotonously strictly increasing sequence.
Effectively, take any $i\in\N$, suppose that $\phiab^{-1}(\sui i) = p_1$ and $\phiab^{-1}(\svi i) = p_2$. 
Both are necessarily positive.
Now, either $\sui {i+1} = \sui i \star \svi i$ and $\svi {i+1} = \svi i$ or $\sui {i+1} = \sui i$ and $\svi {i+1} = \sui i \star \svi i$.
Further, $\sui i \star \svi i = \phiab(p_1)\star\phiab(p_2) = \phiab(p_1+p_2)$. 
As $p_1+p_2 > p_1$ and $p_1+p_2 > p_2$ then, $\pi_{i+1} = p_1+p_2 > \pi_i = \sup\lbrace p_1,p_2\rbrace$.

Now, as $\su$ is monotonously decreasing, it is sufficient to prove that there is no $p$, $\pi_i< p < \pi_{i+1}$ such that $\phiab(p)< \sui i$ whatever $i$ is.
Indeed, take any $i\in\N$. 
Suppose that $\sui i$ is the last minimum record holder found.
If for any $\pi_i < p < \pi_{i+1}$, $\phiab(p)\geq \sui i$, then whatever the case $\sui {i+1} = \sui i$ or $\sui {i+1} < \sui i$ is, there is no minimum record holder $f$ such that $\pi_i<\phiab^{-1}(f)<\pi_{i+1}$. 
In the second of the two possible cases mentioned above, $\sui {i+1}$ is the following minimum record holder under the above condition.

By the same reasoning, we can prove that it is sufficient that for all $i\in\N$ and all $p\in\N$, $\pi_i< p < \pi_{i+1}$, none of the $\phiab(p) > \svi i$.

Prove by induction that for all $i\in\N$ and all $p\in\N$ such that $\pi_i< p < \pi_{i+1}$, $\sui i < \phiab(p) < \svi i$.

\paragraph{Base cases} As $\phiab(1)$ is the first member of $\Fab$ we consider, let it be the first minimum record holder and the first maximum one.
So $\pi_0 = 1$.
The only possible counter-example is $\phiab(0)=1$ but it is out of purpose here.

Now, the next member that we consider is $\phiab(1)\star \phiab(1) = \phiab(2)$ by \refl{lem:next_rh_is_ui_star_vi}.
It may be either $\sui 1$ or $\svi 1$.

Firstly, suppose that $\phiab(1)\phiab(1)<a$. 
So, $\svi 1 = \phiab(2)$ and we start a phase where $\sv$ evolves leaving $\su$ unchanged. 
There is no integer $p$ verifying $1<p<2$. Thus there cannot exist another record holder of the form $\phiab(p)$, $\pi_0<p<\pi_1$.

Further, as long as the current phase runs, $\sui i = \sui 0$ and $\svi i = \svi {i-1}\star\sui 0$ inferring that 
$\pi_i = i+1$.
Suppose now that at a given rank $P$, $a<\sui 0 \svi P$.
This rank exists by \refl{lem:u_v_both_change}.
Then $\sui {P+1} = \sui 0 \star \svi P = \phiab(P+2)$ and $\pi_{P+1} = P+2$.
So, from $p=1$ to $p=P+1$ all the $\phiab(p)$ are maximum record holders. 
They are all retained in $\svi {0}, \cdots, \svi {P}$.
Furthermore, $\phiab(P+2)$ is the first minimum record holder different from $\phiab(1)$.
It is retained in $\sui {P+1}$. 
As a consequence, there is no possibility of a missed record holder $\phi(p)$ where $1\leq p \leq P+2$.

Secondly, in the case that $a<\phiab(1)\phiab(1)$, starts a phase modifying $\su$ instead of $\sv$ and the same reasoning leads to a similar result. 

Let us say that this phase stops at rank $P$, then all the $\phiab(p)$, $1\leq p\leq P+1$ are kept in $\sui {0}, \cdots, \sui {P}$. The first maximum record holder is $\phiab(P+2)=\svi {P+1}$.
In this case also, there are no missed record holder.

\paragraph{Induction step}Let us make use of $P$ as it is defined in the above base case. Assume that for a given $i\in\N, i>P+1$ there is no missed record holder $\phiab(p)$ where $p<\pi_i$ and prove that there is no record holder of the form $\phiab(p)$ where $p$ satisfies $\pi_i<p<\pi_{i+1}$.
 
Let us reason by case and by contradiction in each case. 

\subparagraph{Minimum record holders.}
Suppose their exist $\pi_i<p<\pi_{i+1}$ so that $\phiab(p)<\sui i$. Deduce from this hypothesis that their exists a missed minimum record holder $\phiab(p')$ with $1\leq p' < \pi_i$.

Let $k$ be the largest positive integer such that $\sui {i-k} = \sui i$.
Note that $k$ may be null when $\sui {i-1}\neq \sui i$.
By the construction of the sequence $\s$, the choice of $k$ infers that $\sui {i-k-1}\neq \sui {i-k}$.
Then, we have 
\begin{equation}
\label{eq:u_i-k_fct_u_i-k-1}
\sui {i-k} = \sui {i-k-1} \star \svi {i-k-1} = \sui {i-k-1}\cdot \frac{\svi {i-k-1}}{a}
\end{equation}
while
\begin{equation}
\label{eq:v_i-k_fct_v_i-k-1}
\svi {i-k-1} = \svi {i-k}.
\end{equation}
The definition of the sequence $\pi$ infers that $\phiab^{-1}(\svi i) \leq \pi_i<p$.

The definition of $\phiab$ says that 
\begin{equation}
\label{eq:phi(p)_as_fraction}
\phiab(p) = \frac{a^{p+\dab(p)}}{b^p}.
\end{equation}
By \refl{lem:s_in_FxF}, $\svi i\in\Fab$.
Define $p_1 = \phiab^{-1}(\svi i)$.
Thus $\svi i = \phiab(p_1)$ and 
\begin{equation}
\label{eq:v_i_as_fraction}
\svi i = \frac{a^{p_1+\dab(p_1)}}{b^{p_1}}.
\end{equation}

Let us have a look at $\frac{a \cdot \phiab(p)}{\svi i}$ and prove that this rational number is a minimum record holder that has not been retained in $\su$.

Firstly, let us state that $\frac{a \cdot \phiab(p)}{\svi i}\in\Fab$.

From \eqref{eq:phi(p)_as_fraction} and \eqref{eq:v_i_as_fraction} we deduce that 
\begin{equation}
\label{eq:fraction_phi(p-p1)}
\frac{a \cdot \phiab(p)}{\svi i} = \frac{a^{p-p_1 + \dab(p)-\dab(p_1)+1}}{b^{p-p_1}}.
\end{equation}
As $p-p_1 > 0$ by hypothesis, $\dab(p)\geq \dab(p_1)$.
Then, both exponents are strictly positive. 
So $\frac{a \cdot \phiab(p)}{\svi i}$ is of the same form as any member of $\Fab$.

Now, $\svi i<a$ because $\svi i\in\Fab$ by \refl{lem:s_in_FxF}.
As a consequence, $\phiab(p)<\frac{a \cdot \phiab(p)}{\svi i}$.
As $\phiab(p)\in\Fab$, we obtain that $1<\phiab(p)<\frac{a \cdot \phiab(p)}{\svi i}$.

\refl{lem:u_i+k_and_v_i+k} claims that $\svi i = \svi {i-k} \sui {i-k}^k$.
This implies that 
\begin{equation}
\label{eq:developed_fraction}
\frac{a \cdot \phiab(p)}{\svi i} = \frac{a \cdot \phiab(p)}{\svi {i-k} \sui {i-k}^k}.
\end{equation}
Also $\sui {i-k}\in\Fab$ so $\sui {i-k}>1$ by \refl{lem:s_in_FxF}.
Taking into account that $k$ may be null $\sui {i-k}^k\geq 1$.
Thus, \eqref{eq:developed_fraction} infers that $\frac{a \cdot \phiab(p)}{\svi i} \leq \frac{a \cdot \phiab(p)}{\svi {i-k}}$.
Now, by hypothesis $\phiab(p) < \sui i = \sui {i-k}$.
Thus $\frac{a \cdot \phiab(p)}{\svi i} < \frac{a \cdot  \sui {i-k}}{\svi {i-k}}$.
Equation \eqref{eq:u_i-k_fct_u_i-k-1} claims $\sui {i-k} = \sui {i-k-1}\frac{\svi {i-k-1}}{a}$.
As a consequence, $\frac{a \cdot  \sui {i-k}}{\svi {i-k}} = \sui {i-k-1}$ because $\svi {i-k}=\svi {i-k-1}$ by \eqref{eq:v_i-k_fct_v_i-k-1}.
Then 
\begin{equation}
\label{eq:phiab_p_lt_u_i-k-1}
\frac{a \cdot \phiab(p)}{\svi i} < \sui {i-k-1} < 2.
\end{equation}

In conclusion of this development $1< \frac{a \cdot \phiab(p)}{\svi i}<2$.
It is also of the form of any member of $\Fab$.
Then $\frac{a \cdot \phiab(p)}{\svi i}\in\Fab$ and, by \eqref{eq:phiab_p_lt_u_i-k-1},
\begin{equation}
\label{eq:fraction_lower_than_u_i-k-1}
\frac{a \cdot \phiab(p)}{\svi i} < \sui {i-k-1}.
\end{equation}
As $\su$ monotonously decreases, for all $j\leq i-k-1$, $\frac{a \cdot \phiab(p)}{\svi i} < \sui j$.

Now let us look at the position of $\frac{a \cdot \phiab(p)}{\svi i}$ in the sequence $\su$.

Equation \eqref{eq:fraction_phi(p-p1)} affirms that $\frac{a \cdot \phiab(p)}{\svi i} = \frac{a^{p-p_1 + \dab(p)-\dab(p_1)+1}}{b^{p-p_1}}$ and we know that $\frac{a \cdot \phiab(p)}{\svi i}\in\Fab$.
This leads to $\frac{a \cdot \phiab(p)}{\svi i} = \phiab(p-p_1)$ with $\dab(p-p_1) = \dab(p)-\dab(p_1)+1$ by identification.

Now, $\pi_{i+1} = \phiab^{-1}(\sui i) + \phiab^{-1}(\svi i)$.
Define $p_2 = \phiab^{-1}(\sui i)$.
Then, $\pi_{i+1} = p_1 + p_2$.
Furthermore, by hypothesis, $p<\pi_{i+1}$ which reads $p < p_1+p_2$.
So $\phiab^{-1}(\frac{a \cdot \phiab(p)}{\svi i}) = p-p_1 < p_1+p_2-p_1 = p_2$.
By the definition of $k$, $\sui {i-k} = \sui i$.
As a consequence $\phiab^{-1}(\sui {i-k})=p_2$ and \eqref{eq:fraction_lower_than_u_i-k-1} claims $\frac{a \cdot \phiab(p)}{\svi i} < \sui {i-k-1}$.

In Summary, $\frac{a \cdot \phiab(p)}{\svi i}\in\Fab$, $\frac{a \cdot \phiab(p)}{\svi i} < \sui {i-k-1}$ by \eqref{eq:fraction_lower_than_u_i-k-1}.
Moreover, $\phiab^{-1}(\frac{a \cdot \phiab(p)}{\svi i}) <  \phiab^{-1}(\sui {i-k})$.

This infers that $\frac{a \cdot \phiab(p)}{\svi i}$ is a minimum record holder before $\sui {i-k}$ that has not be retained in $\su$ which contradicts the hypothesis of induction.
Thus, such a $\phiab(p)$ does not exist and there is no minimum record holder $\phiab(p)$ with $p$ strictly comprise between $\pi_i$ and $\pi_{i+1}$.

\subparagraph{Maximum record holders.}
As the proof of the validity of \reft{thm:record_holders} in this case is very similar to the proof of the previous case, we only highlight the main lines.

Introduce a proof by contradiction and assume that there exist $p\in\N$, $\pi_i < p < \pi_{i+1}$ so that $\phiab(p) > \svi i$.

Define $k$ as the largest integer so that $\svi i = \svi {i-k}$. 
As a consequence $\svi {i-k} = \sui {i-k-1} \star \svi {i-k-1} = \sui {i-k-1}\cdot \svi {i-k-1}$.
Also, $\sui {i-k-1} = \sui {i-k}$ and $\sui i = \sui {i-k} \left(\frac{\svi {i-k}}{a}\right)^k$.

Let $p_1$ denotes $\phiab^{-1}(\sui i)$. 
We have $p_1\leq\pi_i<p$.

Now, consider $\frac{\phiab(p)}{\sui i}$. 
We have $\frac{\phiab(p)}{\sui i} = \frac{a^{p-p_1+\dab(p)-\dab(p_1)}}{b^{p-p_1}}$ inferring that is of the same form as any $\Fab$ member.
As $\sui i > 1$, $\frac{\phiab(p)}{\sui i}<\phiab(p)<a$.
Also, $\frac{\phiab(p)}{\sui i} = \frac{\phiab(p)}{\sui {i-k}}\left(\frac{a}{\svi {i-k}}\right)^k$. 
As $\frac{a}{\svi {i-k}}>1$ and $k\geq 0$, $\frac{\phiab(p)}{\sui i} \geq  \frac{\phiab(p)}{\sui {i-k}}$.
Now, by hypothesis $\phiab(p) > \svi i$ so $\frac{\phiab(p)}{\sui i}>\frac{\svi i}{\sui {i-k}}$.
To finish this part, $\sui {i-k} = \sui {i-k-1}$, $\svi i = \svi {i-k}$ and $\svi {i-k} = \sui {i-k-1}\cdot \svi {i-k-1}$ so $\frac{\phiab(p)}{\sui i}>\svi {i-k-1}>1$. 
In conclusion, $\frac{\phiab(p)}{\sui i}\in\Fab$.

In summary, we have $\frac{\phiab(p)}{\sui i}\in\Fab$ and $\frac{\phiab(p)}{\sui i} > \svi {i-k-1}$.
Now, $\sv$ is a monotonously increasing sequence so that for every $j\leq i-k-1$, $\frac{\phiab(p)}{\sui i} > \svi j$.

Define now $p_2 = \phiab^{-1}(\svi i)$. 
With the same logic as in the previous case we obtain that $\phiab^{-1}(\frac{\phiab(p)}{\sui i}) < p_2 = \phiab^{-1}(\svi {i-k})$.

The conclusion of this reasoning is that $\frac{\phiab(p)}{\sui i}$ is a maximum record holder coming before $\svi {i-k}$ and this record holder has not been retained in $\sv$. This is in contradiction with the hypothesis of induction.

Then, such a $\phiab(p)$ does not exist.

So, there is no record holder $\phiab(p)$ with $\pi_i<p<\pi_{i+1}$. Thus, none can be omitted.

In conclusion, we proved by induction that there is no omitted record holder in the sequence $\left(\phiab(p)\right)_{p\in\N-\{0\}}$.
Thus \reft{thm:record_holders} holds.
\end{proof}

In next section we fix the limit of $\su$ and $\sv$.

\section{Limit of $\su$ and $\sv$}
\label{sec:limit_of_u_and_v}
In this section we introduce the limits of $\su$ and $\sv$.

Before this, rephrase a little bit the definition of the convergence of a real sequence in a more convenient manner for our purpose.
\begin{prop}
\label{prop:sequence_convergence}
For any sequence $u=(u_i)_{i\in\N}$ of real numbers, $u$ converges to $l\in\R,~l>0$ if and only if for all $\varepsilon>0$ there exist $N\in\N$ so that for every $i>N,~\frac{1}{1+\varepsilon} l < u_i < l\cdot(1+\varepsilon)$.
\end{prop}

\begin{proof}
Firstly, let us prove that the condition is necessary.

Take any sequence $u=(u_i)_{i\in\N}$ of real numbers that converges to $0<l\in\R$.
Take any $\varepsilon>0$ and define $\varepsilon_1$ such that  $1-\frac{\varepsilon_1}{l}= \frac{1}{1+\varepsilon}$.
This gives $\varepsilon_1 = \frac{l\varepsilon}{1+\varepsilon}$ so $\varepsilon_1>0$ as $l$ is supposed to be positive.

By hypothesis, $u$ converges to $l$. 
This means that there exist $N_1$ so that for every $i>N_1,$ $l-\varepsilon_1 < u_i$. 
This inequality rewrites $l (1-\frac{\varepsilon_1}{l}) < u_i$. That is $\frac{1}{1+\varepsilon} l < u_i$ by the definition of $\varepsilon_1$.

Now, $l\varepsilon > 0$ infers that there exist $N_2$ such that for every $i>N_2,~u_i < l+l\varepsilon = l(1+\varepsilon)$.

Define $N = \sup\{N_1,N_2\}$. Then for every $i>N$, $\frac{1}{1+\varepsilon} l < u_i < l(1+\varepsilon)$.
Thus, for every $\varepsilon>0$, there exist $N\in\N$ such that for every $i>N$, $\frac{1}{1+\varepsilon} l < u_i < l(1+\varepsilon)$.
That is this condition is necessary.

Secondly, let us prove that the condition is sufficient.

Take a sequence $u$ and $l>0$ satisfying that for all $\varepsilon>0$ there exist $N\in\N$ so that for every $i>N,~\frac{1}{1+\varepsilon} l < u_i < l(1+\varepsilon)$ and prove that $u$ converges to $l$.

Take any $\varepsilon>0$ and define $\varepsilon_1 = \frac{\varepsilon}{l}$.
As $l>0$ by hypothesis, $\varepsilon_1 > 0$.

By hypothesis, there exist $N\in\N$ so that for every $i>N,~\frac{1}{1+\varepsilon_1} l < u_i < l(1+\varepsilon_1)$.
Take such a $N$ and consider any $i>N$.
Then $\frac{1}{1+\varepsilon_1} l < u_i < l(1+\varepsilon_1)$.

As $1-\varepsilon_1^2 < 1$, the inequality  $1-\varepsilon_1 < \frac{1}{1+\varepsilon_1}$ holds.
So, $l (1-\varepsilon_1) < u_i < l(1+\varepsilon_1)$ which reads $l-l\varepsilon_1 < u_i < l+l\varepsilon_1$ or more precisely $l-\varepsilon < u_i < l+\varepsilon$ by the definition of $\varepsilon_1$.
In summary, $u$ converges to $l$.

So the condition of \refp{prop:sequence_convergence} is sufficient for $u$ to converge to $l$.

In conclusion, \refp{prop:sequence_convergence} holds.
\end{proof}

As a remark, the condition $l\neq 0$ is strictly necessary otherwise we should get $0<u_i<0$ which is impossible.
Even  with the use of non strict inequalities, we should get $0\leq u_i\leq0$ which limits $u$ to be  the constant sequence $u_i=0$ from a given rank $N$.
Moreover, there may exist a generalization of this property for any $l\neq 0$.
But, as \refp{prop:sequence_convergence} is sufficient for our purpose, we let this property as it is.

This section aims to prove the following theorem.

\begin{theorem}
\label{thm:u_tendsto_1_v_to_2}
For any co-prime integers $a$ and $b$ satisfying $1<a<b$, the sequences $\su$ and $\sv$ defined in \refd{def:define_s} both converges and
$$\lim\limits_{i\to+\infty}\sui i = 1 \textrm { and } \lim\limits_{i\to+\infty}\svi i = a.$$
\end{theorem}

To reach this goal, we have to prove several lemmas.

We know that $\s$ is divided into phases. 
During each of them either $\su$ is modified and $\sv$ remains constant or
$\sv$ changes and $\su$ is left the same.
These two features appear alternatively. 
Understand that a phase where $\su$ changes is followed by another one that modifies $\sv$ itself followed by a phase acting on $\su$ and so on.

Note now that $\s$ may start indifferently by a phase modifying either $\su$ or $\sv$.
For instance, for $a=2$ and $b=3$ the sequence starts with the modification of $\sv$ as seen in \refT{tab:record_holders}. 

The other possibility happens when $a=7$ and $b=8$ for instance. 
Effectively, ${u_{(7,8)}}_0 = {v_{(7,8)}}_0 = \phi_{7,8}(1) = \frac{49}{8}$.
Now $\frac{49}{8}\cdot \frac{49}{8} = \frac{7^4}{8^2} \approx 37.515625 > 7$. So ${u_{(7,8)}}_1 = \frac{7^3}{8^2}$ and ${v_{(7,8)}}_1 = {v_{(7,8)}}_0$.

A consequence of this presentation of the results is that $u$ and $v$ evolve asynchronously and it may be quite difficult to manage it.
To define their respective limit, we propose to extract a subsequence from each one so that these two new sequences evolve synchronously.

Let $\tu$ denotes the subsequence of $\su$ and $\tv$ the one for $\sv$.
To do so, we index each phase by its order of appearance in $\s$.

Remind now that the phase of $\s$ indexed by 1 may modify either $\su$ or $\sv$ depending on $\left(\phiab(1)\right)^2$ compared to $a$.
But all the phases with an odd index modify the same sequence as the first one.
On the opposite, all the phases with an even index change the other sequence.
We call an odd[resp. even] phase, a phase with an odd[resp. even] index.

Describe now how $\tu$ and $\tv$ are extracted from $\su$ and $\sv$.
We keep the first member of each sequence : $\tui 0 = \tvi 0 = \phiab(1)$.
Then we retain only the last element of each phase. 

Thus, if $\left(\phiab(1)\right)^2<a$, the first phase and all the odd ones modify $\sv$ while even ones modify $\su$.
So, when $\left(\phiab(1)\right)^2<a$, for any $i>0$, $\tvi i$ receives the last element of the phase $2i-1$ and $\tui i$ is the last element of the phase $2i$.

Reciprocally, when $a<\left(\phiab(1)\right)^2$, all the odd phases modify $\su$ while the even ones changes $\sv$.
So, when $a<\left(\phiab(1)\right)^2$, for any $i>0$, $\tui i$ receives the last element of the phase $2i-1$ and $\tvi i$ is the last element of the phase $2i$.

This may be summarized as: a pair composed of an odd phase indexed $2i-1$ and the following one indexed $2i$ defines the pair $(\tui i, \tvi i)$. Thus, $\tu$ and $\tv$ may be seen as synchronously defined.

Let us introduce some more notations and give some more precisions related to phases.
The notation $h(\eta)$ identifies the index of head of the phase $\eta$, its first element.
The index of its last element called its tail, is denoted by $t(\eta)$.
The last element of phase $\eta$ is also the first element of phase $\eta+1$.
So two consecutive phases overlap by one element and $t(\eta) = h(\eta+1)$.
Moreover, $\lambda(\eta) = t(\eta)-h(\eta)$ counts the number of elements the phase $\eta$ generates that is its length minus 1.

Let us describe what is $\lambda(\eta)$.
\begin{lem}
\label{lem:def_lambda}
\begin{equation*}
\lambda(\eta) = \left\lbrace \begin{array}{ll}
\left \lfloor \frac{\log a - \log \svi {h(\eta)}}{\log \sui {h(\eta)}} \right\rfloor & \textrm{ when phase $\eta$ modifies $\sv$,} \\[1em]
\left \lfloor \frac{\log \sui {h(\eta)}}{\log a - \log \svi {h(\eta)}} \right\rfloor & \textrm{ when phase $\eta$ modifies $\su$.}
\end{array}\right.
\end{equation*}
\end{lem}

\begin{proof}
Firstly, let us suppose that phase $\eta$ modifies $\sv$. 
This means that for every $0 \leq k \leq \lambda(\eta)$, $\sui {h(\eta) + k} = \sui {h(\eta)}$ and $\svi {h(\eta) + k} = \svi {h(\eta)} \sui {h(\eta)}^k$ by \refl{lem:u_i+k_and_v_i+k}.
Furthermore, phase $\eta$ modifies $\sv$ so for every $0 \leq k < \lambda(\eta)$, $\sui {h(\eta)+k} \svi {h(\eta)+k}<a$ by the construction of $\s$.
Replacing the two factors by their respective value we get $\sui {h(\eta)} \svi {h(\eta)} \sui {h(\eta)}^k<a$ which may be rewritten as $\svi {h(\eta)} \sui {h(\eta)}^{k+1}<a$.
This is verified particularly for $k = \lambda(\eta)-1$.
Thus, $\svi {h(\eta)} \sui {h(\eta)}^{\lambda(\eta)}<a$.
Now, phase $\eta$ stops at $t(\eta) = h(\eta)+\lambda(\eta)$.
This infers that $a<\svi {h(\eta)} \sui {h(\eta)}^{\lambda(\eta)+1}$.
In summary, $$\svi {h(\eta)} \sui {h(\eta)}^{\lambda(\eta)}<a<\svi {h(\eta)} \sui {h(\eta)}^{\lambda(\eta)+1}.$$
The natural logarithm of this relationship gives $$\log \svi {h(\eta)} + \lambda(\eta) \log \sui {h(\eta)}<\log a<\log \svi {h(\eta)} + (\lambda(\eta)+1) \log \sui {h(\eta)}.$$
Some basic manipulations lead to $$\lambda(\eta) < \frac{\log a-\log \svi {h(\eta)}}{\log \sui {h(\eta)}} <\lambda(\eta)+1.$$
This may be rewritten  $$\frac{\log a-\log \svi {h(\eta)}}{\log \sui {h(\eta)}} -1 <\lambda(\eta) <\frac{\log a-\log \svi {h(\eta)}}{\log \sui {h(\eta)}}.$$
That is $\lambda(\eta) = \left \lfloor \frac{\log a - \log \svi {h(\eta)}}{\log \sui {h(\eta)}} \right\rfloor$ which is the expected result.

Secondly, let us suppose that phase $\eta$ modifies $\su$. 
The sequence $\sv$ does not change so for every $0 \leq k \leq \lambda(\eta)$, $\svi {h(\eta) + k} = \svi {h(\eta)}$. 
Now, \refl{lem:u_i+k_and_v_i+k} claims that for every $0 \leq k \leq \lambda(\eta)$, $\sui {h(\eta) + k} = \sui {h(\eta)} \left(\frac{\svi {h(\eta)}}{a}\right)^k$.
Phase $\eta$ modifies $\su$ so for every $0 \leq k < \lambda(\eta)$, $a<\sui {h(\eta)} \left(\frac{\svi {h(\eta)}}{a}\right)^{k}\svi {h(\eta)}$ by the definition of $\s$.
This may be rewritten $1<\sui {h(\eta)} \left(\frac{\svi {h(\eta)}}{a}\right)^{k+1}$
This is particularly the case for $k = \lambda(\eta)-1$.
Thus, $1<\sui {h(\eta)} \left(\frac{\svi {h(\eta)}}{a}\right)^{\lambda(\eta)}$.
Phase $\eta$ stops at $t(\eta) = h(\eta)+\lambda(\eta)$ implying that $\sui {h(\eta)} \left(\frac{\svi {h(\eta)}}{a}\right)^{\lambda(\eta)}\svi {h(\eta)}<a$.
This reads $\sui {h(\eta)} \left(\frac{\svi {h(\eta)}}{a}\right)^{\lambda(\eta)+1}<1$.
The two obtained inequalities may be summarized as $$\sui {h(\eta)} \left(\frac{\svi {h(\eta)}}{a}\right)^{\lambda(\eta)+1}<1<\sui {h(\eta)} \left(\frac{\svi {h(\eta)}}{a}\right)^{\lambda(\eta)}.$$
The application of the natural logarithm function to that double inequality gives 
\begin{eqnarray*}
\log \sui {h(\eta)} + (\lambda(\eta)+1) (\log \svi {h(\eta)} - \log a) < 0\\
 <\log \sui {h(\eta)} + \lambda(\eta)(\log \svi {h(\eta)} - \log a).
\end{eqnarray*}
Or 
$$(\lambda(\eta)+1) (\log \svi {h(\eta)} - \log a) <- \log \sui {h(\eta)} < \lambda(\eta)(\log \svi {h(\eta)} - \log a).$$
The three expressions are negative then we rewrite this inequality as
$$\lambda(\eta) (\log a - \log \svi {h(\eta)}) <\log \sui {h(\eta)} < (\lambda(\eta)+1)(\log a -\log \svi {h(\eta)}).$$
Lastly, a division by $\log a -\log \svi {h(\eta)}$ leads to
$$\lambda(\eta)<\frac{\log \sui {h(\eta)}}{ \log a - \log \svi {h(\eta)} }<\lambda(\eta)+1.$$
That is $\frac{\log \sui {h(\eta)}}{ \log a - \log \svi {h(\eta)} }-1<\lambda(\eta)<\frac{\log \sui {h(\eta)}}{ \log a - \log \svi {h(\eta)} } $ 
or 
$$\lambda(\eta) = \left\lfloor \frac{\log \sui {h(\eta)}}{ \log a - \log \svi {h(\eta)} } \right\rfloor.$$
This is what \refl{lem:def_lambda} affirms in the second case.
\end{proof}

In-fine, another manner to explain \refl{lem:def_lambda} may be  $\lambda(k)$ is the highest integer so that $(u_{h(k)})^{\lambda(k)}v_{h(k)}<a$ in the first case and, in the second case, $\lambda(k)$ is the highest integer such that $1<u_{h(k)}\left(\frac{v_{h(k)}}{a}\right)^{\lambda(k)}$.

Next lemma is a direct application of \refl{lem:u_i+k_and_v_i+k} to explain the relationship between the head of a phase and its tail.
\begin{lem}
\label{lem:head_and_tail_of_a_phase}
Consider any phase indexed by $\eta$ of $\s$. 
If this phase modifies $\sv$, then $\sui {t(\eta)} = \sui {h(\eta)}$ and $\svi {t(\eta)} = (\sui {h(\eta)})^{\lambda(\eta)} \svi {h(\eta)}$. 
Reciprocally, if phase $\eta$ modifies $\su$, then 
$\svi {t(\eta)} = \svi {h(\eta)}$ and $\sui {t(\eta)} = \sui {h(\eta)} \left(\frac{\svi {h(\eta)}}{a}\right)^{\lambda(\eta)}$.
\end{lem}

\begin{proof}
As announced above, this lemma is an immediate consequence of \refl{lem:u_i+k_and_v_i+k}.
Indeed, by the definition of a phase, only one of the two sequences $\sv$ or $\su$ changes.
Thus each of the results described in \refl{lem:head_and_tail_of_a_phase} are just the particular case of \refl{lem:u_i+k_and_v_i+k} when $k = \lambda(\eta)$.
\end{proof}

The following lemma gives a recursive definition of $\tu$ and $\tv$. 
These two sequences are so inter-dependant, that it is quite impossible to define one without defining the second one.

\begin{lem}
\label{lem:def_tu_and_tv}
Given $1<a<b$ two co-prime integers, the sequences $\tu$ and $\tv$ may be recursively defined by: $\tui 0 = \tvi 0 = \phiab(1)$ and for any $0<i$, 
\begin{itemize}
\item if $(\phiab(1))^2 < a$ then
\begin{equation*}
\begin{array}{l}
\tvi {i+1} = \tvi i \cdot (\tui i)^{\lambda_1}\\
\tui {i+1} = \tui i \cdot \left(\frac{\tvi {i+1}}{a}\right)^{\lambda_2}
\end{array}
\end{equation*}
where $\lambda_1$ is the highest positive integer such that $\tvi i \cdot (\tui i)^{\lambda_1}<a$ and $\lambda_2$ is the highest positive integer such that $1<\tui i \cdot \left(\frac{\tvi {i+1}}{a}\right)^{\lambda_2}$.
\item if $a<(\phiab(1))^2 $ then
\begin{equation*}
\begin{array}{l}
\tui {i+1} = \tui i \cdot \left(\frac{\tvi {i}}{a}\right)^{\lambda_1}\\
\tvi {i+1} = \tvi i \cdot (\tui {i+1})^{\lambda_2}
\end{array}
\end{equation*}
where $\lambda_1$ is the highest positive integer such that $1<\tui i \cdot \left(\frac{\tvi {i}}{a}\right)^{\lambda_1}$ and $\lambda_2$ is the highest positive integer such that $\tvi i \cdot (\tui {i+1})^{\lambda_2}<a$.
\end{itemize} 
\end{lem}

Even if the statement of this lemma may seem cumbersome, its proof is no less simple.

Describe it before we start its validation.
Since the first phase of the sequence $\s$ may modify either $\sv$ or $\su$ depending on  $(\phiab(1))^2$ compared to $a$, we have two possible behaviours.
Consequently, \refl{lem:def_tu_and_tv} consider these two cases separately.
In the first case, $\sv$ changes on every odd phases while $\su$ is modified during even ones.
So, \refl{lem:def_tu_and_tv} first express $\tvi {i+1}$ from $\tvi i$ and $\tui i$.
Then it gives the expression of $\tui {i+1}$ from $\tui i$ and $\tvi {i+1}$.
In the second case, it is the opposite.
That is every odd phases modify $\su$ when even ones act on $\sv$.
\refl{lem:def_tu_and_tv} reflects this behaviour by giving first the expression of $\tui {i+1}$ from $\tui i$ and $\tvi {i}$.
Then it describes how to obtain $\tvi {i+1}$ from $\tvi i$ and $\tui {i+1}$.

Now, let us introduce the proof of \refl{lem:def_tu_and_tv}.

\begin{proof}
Let us consider the two possible cases separately.
\subparagraph{First Case: $(\phiab(1))^2 < a$}
This condition infers that all odd phases of $\s$ modify $\sv$ while all the even ones change $\su$.
As a consequence, for every $i\in\N-\{0\}$, $\tvi i$ is the last element of phase $2i-1$ and $\tui i$ is the last element of phase $2i$.
This may be described by $\tvi i = \svi {t(2i-1)}$ and $\tui i = \sui {t(2i)}$.

Now, the phases are build so that $t(\eta) = h(\eta+1)$.
This infers that $\svi {t(2i-1)} = \svi {h(2i)}$.

Then, phase $2i$ does not change $\sv$ so $\svi {t(2i)} = \svi {t(2i-1)}$.
Here too the construction of phases applies. 
So $\svi {h(2i+1)} = \svi {t(2i-1)}$ and $\sui {h(2i+1)} = \sui {t(2i)}$.

\refl{lem:head_and_tail_of_a_phase} claims that $\svi {t(2i+1)} = (\sui {h(2i+1)})^{\lambda(2i+1)} \svi {h(2i+1)}$.
More, $\svi {t(2i+1)} = \svi {t(2i-1)} = \tvi i$ and $\sui {h(2i+1)} = \sui {t(2i)} = \tui i$.
Use these equalities in previous equation to obtain  $\svi {t(2i+1)} = (\tui i)^{\lambda(2i+1)} \tvi i$.
Note that $\svi {t(2i+1)} = \tvi {i+1}$ so that $\tvi {i+1} = (\tui i)^{\lambda(2i+1)} \tvi i$.

The definition of $\lambda(2i+1)$ plus \refl{lem:def_lambda} give that $\lambda(2i+1)$ is the highest positive integer such that $(\sui {h(2i+1)})^{\lambda(2i+1)} \svi {h(2i+1)}<a$. 
Then after the substitutions of $\sui {h(2i+1)}$ by  $\tvi {i+1}$ and of $\svi {h(2i+1)}$  by $\tui i$ we obtain that $\lambda_1 = \lambda(2i+1)$.

Now, during phase $2i+1$ $\su$ remains constant. So $\sui {t(2i+1)} = \sui {h(2i+1)}$.
Furthermore, the head of phase $2i+2$ is the tail of phase $2i+1$. 
Thus, \refl{lem:head_and_tail_of_a_phase} infers that $\sui {t(2i+2)} = \sui {h(2i+2)} \left(\frac{\svi {h(2i+2)}}{a}\right)^{\lambda(2i+2)}$. 
The application of the appropriated substitutions as for $\sv$ above gives that $\tui {i+1} = \sui {t(2i+2)} = \tui i \left(\frac{\tvi {i+1}}{a}\right)^{\lambda(2i+2)}$.

The same reasoning as for $\lambda_1$ proves that $\lambda_2 = \lambda(2i+2)$.

Thus, the conclusion of \refl{lem:def_tu_and_tv} in this case holds under the conditions it defines.

\subparagraph{Second Case: $(\phiab(1))^2 < a$}
The same method of proof as in the first case applies also here. 
We have just to correctly reorder the two phases.
The phase $2i-1$ alters sequence $\su$ which gives the expression of $\tui {i+1}$ from $\tui i$ and $\tvi i$ given in \refl{lem:def_tu_and_tv}. 
In this case, also, $\lambda_1 = \lambda(2i+1)$.
Follows a phase that modifies $\sv$ giving $\tvi {i+1}$ as a function of $\tui {i+1}$ and $\tvi i$ as introduced in \refl{lem:def_tu_and_tv}.
Here too, $\lambda_2 = \lambda(2i+2)$.

To conclude that \refl{lem:def_tu_and_tv} holds.
\end{proof}

\refl{lem:def_tu_and_tv} gives a description of $\tu$ and $\tv$ that infers that for each $i\in\N-\{0\}$ $\tui {i+1} \neq \tui i$ and $\tvi {i+1} \neq \tvi i$.
The sequence $\tu$ is extracted from sequence $\su$ which monotonously decreases without any duplication of a term so $\tu$ is a strictly monotonously decreasing sequence.
More, $\tu$ converges because it also admit 1 as a lower bound and it shares the same limit as $\su$.
By the same reasoning, $\tv$ is a strictly monotonously increasing sequence which converges to the same limit as $\sv$.

In the sack of simplicity, let us symmetrize the problem so that we have to deal with two sequences having the same behaviours.

Firstly, observe that $\tvi {i+1}$ is obtained by the product of $\tvi i$ by a power of $\tui i$ or $\tui {i+1}$ following the case.
But, on its side, $\tui {i+1}$ is obtained by the product of $\tui i$ by a power of $\frac{\tvi i}{a}$ or $\frac{\tvi {i+1}}{a}$ depending on the case.
So the idea to manage the sequence $\frac{\tv}{a}$ instead of $\tv$.
Note that, $\frac{\tv}{a}$ is also a strictly monotonously increasing sequence but its upper bound is 1 instead of $a$.

Further, all the terms of $\frac{\tv}{a}$ are lower than one but they are all different from 0.
So the sequence $\frac{a}{\tv}$ contains only finite rational numbers.
Now, as $\tv$ is an strictly monotonously increasing sequence, $\frac{a}{\tv}$ is a strictly monotonously decreasing one.
More, $\frac{\tv}{a}$ admits 1 as a upper bound so, $\frac{a}{\tv}$ admits 1 as a lower bound.

Let us call $\tw = \frac{a}{\tv}$ this sequence and highlight how $\tu$ and $\tw$ interact.

\begin{lem}
\label{lem:def_tu_and_tw}
Given $1<a<b$ two co-prime integers, the sequences $\tu$ and $\tw$ may be recursively defined by $\tui 0 = \phiab(1)$ and $\twi 0 = \frac{a}{\phiab(1)}$ and for any $0<i$, 
\begin{itemize}
\item if $\tui 0<\twi 0$ then
\begin{equation*}
\begin{array}{l}
\twi {i+1} = \frac{\twi i} {(\tui i)^{\lambda_1}}\\
\tui {i+1} = \frac{\tui i} {(\twi {i+1})^{\lambda_2}}
\end{array}
\end{equation*}
where $\lambda_1$ is the highest positive integer such that $(\tui i)^{\lambda_1}<\twi i$ and $\lambda_2$ is the highest positive integer such that $(\twi {i+1})^{\lambda_2}<\tui i$.
\item if $\twi 0<\tui 0$ then
\begin{equation*}
\begin{array}{l}
\tui {i+1} =  \frac{\tui i} {(\twi {i})^{\lambda_1}}\\
\twi {i+1} = \frac{\twi i} {(\tui {i+1})^{\lambda_2}}
\end{array}
\end{equation*}
where $\lambda_1$ is the highest positive integer such that $(\twi {i})^{\lambda_1}<\tui i$ and $\lambda_2$ is the highest positive integer such that $(\tui {i+1})^{\lambda_2}<\twi i$.
\end{itemize} 
\end{lem}

\begin{proof}
The proof of this lemma relies on \refl{lem:def_tu_and_tv}.
Thus, let us build a proof by cases based on \refl{lem:def_tu_and_tv} statement.

Before going further in the proof, by definition $\tw=\frac{a}{\tv}$ inferring that $\tv = \frac{a}{\tw}$ as none of the $\tvi i$ and $\twi i$ can be null. That is, for every $i\in\N,~\tvi i = \frac{a}{\twi i}$.

The base case may be immediately obtained from the definition of $\tw$. Indeed, $\tu$ is the same between the two lemmas so $\tui 0 = \phiab(1)$. Now, $\tw=\frac{a}{\tv}$ by its definition so $\twi 0 = \frac{a}{\tvi 0} = \frac{a}{\phiab(1)}$.

Now, let us take care of the equivalence of the condition that separates both cases used in \refl{lem:def_tu_and_tv} with the one used in \refl{lem:def_tu_and_tw}.
In \refl{lem:def_tu_and_tv}, the criterion is $(\phiab(1))^2<a$.
We can consider that it is $\tui 0 \tvi 0 < a$ as it was the case for $\su$ and $\sv$.
If we replace $\tvi 0$ by its expression $\frac{a}{\twi 0}$, we get $\tui 0 < \twi 0$ after some manipulations of the inequality. 
This is the criterion of \refl{lem:def_tu_and_tw}.

Reason now by case.

\subparagraph{Case 1: $(\phiab(1))^2 < a$}
In this case \refl{lem:def_tu_and_tv} claims that $$\tvi {i+1} = \tvi i (\tui i)^{\lambda_1}$$ where $\lambda_1$ is the highest positive integer such that $\tvi i (\tui i)^{\lambda_1} < a$. 

Focus on the description of $\tvi {i+1}$.
A substitution of $\tv$ by $\frac{a}{\tw}$ in this assertion gives $\frac{a}{\twi {i+1}} = \frac{a}{\twi {i}} (\tui i)^{\lambda_1} $. 
A simplification by $a$ followed by an inversion of the equation lead to $\twi {i+1} =  \frac{\twi {i}}{(\tui i)^{\lambda_1}}$.
This is the expression of $\twi {i+1}$ we are waiting for in this case.

Now, focus on the definition of $\lambda_1$. 
\refl{lem:def_tu_and_tv} defines $\lambda_1$ as the highest positive integer such that $\tvi i (\tui i)^{\lambda_1} < a$.
Replace $\tv$ by $\frac{a}{\tw}$ in the condition to obtain $\frac{a}{\twi {i}} (\tui i)^{\lambda_1} < a$.
A multiplication of this inequality by $\frac{\twi {i}}{a} > 0$ 
leads to  $(\tui i)^{\lambda_1}<\twi i$.
So, $\lambda_1$ is the highest positive integer such that $(\tui i)^{\lambda_1}<\twi i$.
This is the definition of $\lambda_1$ given by \refl{lem:def_tu_and_tw} in this case.

Furthermore, in this case, \refl{lem:def_tu_and_tv} states that $\tui {i+1} = \tui i \cdot \left(\frac{\tvi {i+1}}{a}\right)^{\lambda_2}$ where $\lambda_2$ is the highest positive integer such that $1<\tui i \cdot \left(\frac{\tvi {i+1}}{a}\right)^{\lambda_2}$.

Here too, focus firstly on the definition of $\tui {i+1}$.
Again, substitute $\tv$ by $\frac{a}{\tw}$ in this assertion.
We get $\tui {i+1} = \tui i \cdot \left(\frac{\frac{a}{\twi {i+1}}}{a}\right)^{\lambda_2}$.
We can simplify the fraction by $a$. 
Then a simple rewriting leads to $\tui {i+1} = \frac{\tui i }{(\twi {i+1})^{\lambda_2}}$. 
This is the expression of $\tui {i+1}$ given by \refl{lem:def_tu_and_tw}.

To finish with this case, consider the definition of $\lambda_2$.
\refl{lem:def_tu_and_tv} defines $\lambda_2$ as the highest positive integer such that $1<\tui i \cdot \left(\frac{\tvi {i+1}}{a}\right)^{\lambda_2}$.
Once more, substitute $\tv$ by $\frac{a}{\tw}$ in this assertion.
This gives $1<\tui i \cdot \left(\frac{\frac{a}{\twi {i+1}}}{a}\right)^{\lambda_2}$.
Simplify the fraction then apply some basic rewritings to get $(\twi {i+1})^{\lambda_2}<\tui i$, the condition that defines $\lambda_2$ in \refl{lem:def_tu_and_tw} in this case.

\subparagraph{Case 2: $a<(\phiab(1))^2 < a^2$}
In the same way as for the first case, it is possible to establish conclusions of \refl{lem:def_tu_and_tw} for this case.

Indeed, the intermediate results are the same.
We just need to exchange the indices $i \leftrightarrow i+1$ in the correct places and exchange $\lambda_1 \leftrightarrow \lambda_2$ and we are done.

In conclusion \refl{lem:def_tu_and_tw} holds.
\end{proof}

Now comes a lemma that confirms that both sequences decrease and admit 1 as a lower bound. It also introduces the conservation of the initial order of $\tui 0$ and $\twi 0$.
\begin{lem}
\label{lem:tu_and_tw_decrease}
Given $a$ and $b$ two integers verifying $1<a<b$. Define $\tu$ and $\tw$ as in \refl{lem:def_tu_and_tw}.

For any $i\in\N$, if $\tui 0<\twi 0$ then $$1<\tui {i+1}<\twi {i+1}<\tui i<\twi i$$ else $$1<\twi {i+1}<\tui {i+1}<\twi i<\tui i.$$
\end{lem}

\begin{proof}
As described in \refl{lem:def_tu_and_tw}, the two possible cases are totally symmetric. 
So it is sufficient to build the proof for one of the two cases. 
The proof for the other one is the same with just some minor adaptations.

Both of $\tui 0$ and $\twi 0$ are strictly greater than 1.
Indeed, $\tui 0 = \phiab(1)$ and by the definition of $\phiab$ only $\phiab(0) = 1$ and $\phiab(p)>1$ for any $p\neq 0$.
Concerning $\twi 0 = \frac{a}{\phiab(1)}$, the definition of $\phiab$ infers that $1<\phiab(1)<a$. 
So, $\frac{1}{a}<\frac{1}{\phiab(1)}<1$ that is $1<\frac{a}{\phiab(1)}<a$.

Now, assume that $1<\tui 0 < \twi 0$ and let us prove by induction that for any $i\in\N$, $1<\tui {i+1}<\twi {i+1}<\tui i<\twi i$.

Take any $i\in\N$ and suppose that $1<\tui i < \twi i$. 
Prove that $1<\tui {i+1}<\twi {i+1}<\tui i<\twi i$.

This is the first case of \refl{lem:def_tu_and_tw}. 
The exponent $\lambda_1$ is defined as the highest integer so that $(\tui i)^{\lambda_1}<\twi i$.
This may be written as $(\tui i)^{\lambda_1}<\twi i<(\tui i)^{\lambda_1+1}$.
Divide this by $(\tui i)^{\lambda_1}$ to get $1<\frac{\twi i}{(\tui i)^{\lambda_1}}<\tui i$.
Note that \refl{lem:def_tu_and_tw} says that $\twi {i+1} = \frac{\twi i} {(\tui i)^{\lambda_1}}$ so we get $1<\twi {i+1}<\tui i$.
Now $\lambda_2$ is defined as the highest positive integer such that $(\twi {i+1})^{\lambda_2}<\tui i$.
By the same way, $(\twi {i+1})^{\lambda_2}<\tui i<(\twi {i+1})^{\lambda_2+1}$.
Divide this by $(\twi {i+1})^{\lambda_2}$ to get $1<\frac{\tui i}{(\twi {i+1})^{\lambda_2}}<\twi {i+1}$.
As \refl{lem:def_tu_and_tw} states that $\frac{\tui i}{(\twi {i+1})^{\lambda_2}} = \tui {i+1}$, we obtain
$1<\tui {i+1}<\twi {i+1}<\tui i < \twi i$.

In conclusion of this, it is sufficient that $1<\tui i < \twi i$ to get $1<\tui {i+1}<\twi {i+1}<\tui i < \twi i$.
This conclusion includes $1<\tui {i+1}<\twi {i+1}$, the hypothesis of induction at rank $i+1$.

Now we can apply this to $i=0$ which gives $1<\tui {1}<\twi {1}<\tui 0 < \twi 0$.
By induction we proved the first case of \refl{lem:tu_and_tw_decrease}.

As said above, the proof for the second case is the same.
\end{proof}

Now, both of the sequences strictly decrease and admit a lower bound, so they converge. Let us prove that they share the same limit.

\begin{lem}
\label{lem:tu_and_tw_have_the_same_limit}
For any $1<a<b$ two co-prime integers, the two sequences $\tu$ and $\tw$ have the same limit.
\end{lem}

\begin{proof}
We know that $\tu$ and $\tw$ both converge. Say that $\lim\limits_{i\to+\infty} \tui i = l_1$ and $\lim\limits_{i\to+\infty} \twi i = l_2$.
Let us consider the first case of \refl{lem:def_tu_and_tw}.
In this case, \refl{lem:tu_and_tw_decrease} claims that for any $i\in\N$, $\tui i < \twi i$.
So $\tw$ is an upper bound for $\tu$ inferring that $l_2\geq l_1$.
But, the same lemma also states that for any $i$, $\twi {i+1}<\tui i$ implying that $\tu$ is an upper bound of the subsequence of $\tw$ starting at $i=1$.
As a consequence $l_1\geq l_2$.
Then the only possibility we have is $l_2 = l_1$.

The same method of proof applied to the second case of \refl{lem:def_tu_and_tw} leads to the same result.
So, \refl{lem:tu_and_tw_have_the_same_limit}.
\end{proof}

Now introduce the last lemma of this section.

\begin{lem}
\label{lem:tu_and_tw_converge_to_1}
Let $1<a<b$ be two co-prime integers an $\tu$ and $\tw$ the two sequences defined in \refl{lem:def_tu_and_tw}. Then
$$\lim\limits_{i\to+\infty} \tui i = \lim\limits_{i\to+\infty} \twi i = 1.$$
\end{lem}

\begin{proof}
Let us prove \refl{lem:tu_and_tw_converge_to_1} by contradiction.

Suppose that there exist $l\in\R,~l>1$ so that $\lim\limits_{i\to+\infty} \tui i = \lim\limits_{i\to+\infty} \twi i = l.$
This limit is supposed greater than 1 so that we can define $\delta>0$ such that $l = 1+\delta$.

Consider the first case of \refl{lem:def_tu_and_tw}.
By hypothesis, $\lim\limits_{i\to+\infty} \twi i = l$.
\refp{prop:sequence_convergence} claims that for every $\varepsilon>0$, there exist $N\in\N$ such that for every integer $i>N$, $\twi i < l(1+\varepsilon)$.
Take such an $0<\varepsilon<\delta$.
Take any $N$ so that for every integer $i>N$, $ \twi i < l(1+\varepsilon)$.
Take now any $i>N$.
Thus, $l < \twi i < l(1+\varepsilon)$ because $\tw$ monotonously decreases to $l$.
Now, $\tui i < \twi i$ by \refl{lem:tu_and_tw_decrease}.
More, $\tu$ also decreases and admit $l$ as limit by hypothesis.
Thus $l < \tui i < l(1+\varepsilon)$.
Define now, $\lambda_1$ as the highest integer so that $(\tui i)^{\lambda_1}<\twi i$
So, we have $\frac{1}{(l(1+\varepsilon))^{\lambda_1}}<\frac{1}{(\tui i)^{\lambda_1}}<\frac{1}{l^{\lambda_1}}$.
Finally we obtain $\frac{l}{(l(1+\varepsilon))^{\lambda_1}}<\frac{\twi i}{(\tui i)^{\lambda_1}}<\frac{l(1+\varepsilon)}{l^{\lambda_1}}$.
Essentially, this infers that $\twi {i+1} < \frac{1+\varepsilon}{l^{\lambda_1-1}}$.
In the worst case, $\lambda_1 = 1$ giving that $\twi {i+1} < 1+\varepsilon$ in any case.
As $\varepsilon$ has been chosen so that $\varepsilon<\delta$ thus $\twi {i+1} < 1+\delta = l$.
Now, $\tw$ strictly decreases inferring that $l$ cannot be its limit denying the hypothesis that it is its limit.
So, $\tw$ cannot have a limit $l>1$ but it accepts a limit $l\geq1$. So its limit is 1.
By \refl{lem:tu_and_tw_have_the_same_limit}, we also get that $\lim\limits_{i\to+\infty} \tui i = \lim\limits_{i\to+\infty} \twi i = 1$ in this case, the conclusion of \refl{lem:tu_and_tw_converge_to_1}.

Here too, the same method of proof applies to the second case.

To conclude that \refl{lem:tu_and_tw_converge_to_1} holds.
\end{proof}

From \refl{lem:tu_and_tw_converge_to_1} it is possible to obtain the proof of \reft{thm:u_tendsto_1_v_to_2}.

\begin{proof} [Proof of  \reft{thm:u_tendsto_1_v_to_2}]
Start by the limit of $\su$.
By \refl{lem:tu_and_tw_converge_to_1} $\lim\limits_{i\to+\infty} \tui i = 1$.
Now, \refl{lem:u_and_v_converge} we know that $\su$ also converges.
As $\tu$ is an infinite subsequence of $\su$ both must have the same limit.
Thus, $\lim\limits_{i\to+\infty} \sui i = 1$.

Concerning $\sv$ now.
We have by \refl{lem:tu_and_tw_converge_to_1} $\lim\limits_{i\to+\infty} \twi i = 1$.
By the definition of $\tw$, $\tv = \frac{a}{\tw}$.
The function $x \to \frac{a}{x}$ is continue on $\R-\{0\}$ and especially around 1.
Thus, $\lim\limits_{i\to+\infty} \tvi i = \frac{a}{\lim\limits_{i\to+\infty} \twi i} = a$.
Finally, $\sv$ also converges by \refl{lem:u_and_v_converge}.
In conclusion, $\lim\limits_{i\to+\infty} \svi i = \lim\limits_{i\to+\infty} \tvi i = a$ because $\tv$ is a subsequence of $\sv$.
In summary, the conclusion of \reft{thm:u_tendsto_1_v_to_2} holds.
\end{proof}

\section{$\Fab$ is dense in $[1,a]_\R$}
In this section, we introduce the main theorem of this paper and two of its corollaries.
Let $[1,a]_\R$ denote the set of all real numbers $x$ verifying $1\leq x \leq a$.
We claim the following theorem.
\begin{theorem}
\label{thm:F_dense_in_1_a}
For any co-prime integers $a$ and $b$ verifying $1<a<b$, the set $\Fab$ is dense in $[1,a]_\R$.
\end{theorem}

\begin{proof}
The segment $[1,a]_\R$ with the usual topology of $\R$, is a metric space.
Thus we have to prove that the closure $\overline{\Fab}$ of $\Fab$ is $[1,a]_\R$.

First, prove by contradiction that $\overline{\Fab}\subset [1,a]_\R$.
Suppose their exist $l>a$ that is the limit of any convergent sequence $\sigma$ of $\Fab$.
This hypothesis infers that there exist an infinity of $\sigma_i$ as near as we want of $l$.
So there must exist an infinity of $\sigma_i\in\Fab$ such that $a<\sigma_i<l+\varepsilon$.
By definition of $\Fab$, this is impossible.
So, there is no sequence of $\Fab$ that converges to any $l>2$.

The same reasoning leads to that any $l<1$ cannot be the limit of any sequence of $\Fab$.
Thus $\overline{\Fab} \subset [1,a]_\R$.

Reciprocally, by the definition of the closure of set, $\Fab\subset\overline{\Fab}$.
\reft{thm:u_tendsto_1_v_to_2} infers that $a\in\overline{\Fab}$ and $1\in\Fab\subset \overline{\Fab}$.

Now, take any $x\in ]1,a[_\R,~x\not\in\Fab$ and prove that $x$ may be the limit of at least one sequence of $\Fab$.

Let us build such a sequence $\sigma$.
The real number $x$ verifies $x>1$ so that there exist an infinity of $\sui i$, $1<\sui i<x$ by the fact that $\su$ converges to 1.
We cannot have $\sui i=x$ because $x\not\in\Fab$.
Let us write that $\sigma_{0} = \sup\{\sui i: i\in\N \textrm { and } \sui i<x \}$.
Let us define $\sigma_{i+1}$ from $\sigma_i$ for any $i\in\N$ now.

\paragraph{Base case}
Consider $\frac{x}{\sigma_0}$.
By the construction of $\sigma_0$, $\sigma_0 < x$.
Then $1<\frac{x}{\sigma_0}$. 
By the convergence of $\su$ to 1, there exist an infinity of $\sui i<\frac{x}{\sigma_0}$.
Define $i_1 = \inf\{i\in\N: \sui i<\frac{x}{\sigma_0} \}$.
Then $\sigma_0\star\sui {i_1}\in\Fab$ by the definition of $\star$.
$\sigma_0\cdot\sui {i_1} < \sigma_0 \cdot \frac{x}{\sigma_0}$ by the definition of $i_1$.
Thus $\sigma_0\cdot\sui {i_1} < x < a$ so $\sigma_0\star\sui {i_1} = \sigma_0\cdot\sui {i_1}$.
More, as $\sui {i_1}>1$, $\sigma_0<\sigma_0\cdot\sui {i_1}$.
Define $\sigma_1 = \sigma_0\cdot\sui {i_1}$.

The same process may be iterated to define $i_2$ and $\sigma_2$ from $\sigma_1$, then $i_3$ and $\sigma_3$ from $\sigma_2$, $\cdots $, $i_{j+1}$ and $\sigma_{j+1}$ from $\sigma_j$, and so on.

The sequence $\sigma$ build such a way is a strictly monotonously increasing sequence. More, it accepts $x$ as an upper bound. 
So $\sigma$ converges to a limit lower than or equal to $x$.

Now, suppose that $\sigma$ converges to a limit $l<x$ and define $\delta>0$ such that $\frac{x}{l} = 1 + \delta$.
Take any $\varepsilon, ~0<\varepsilon<\delta$.
By \refp{prop:sequence_convergence} and the hypothesis that $\sigma$ converges to $l$, there exist $N\in\N$ such that for every $n>N$, $\frac{1}{1+\varepsilon} l< \sigma_n$.
Take the lowest $n$ verifying $\frac{1}{1+\varepsilon} l < \sigma_n$ and have a look at the following terms of $\sigma$.

The construction of $\sigma$ defines $i_{n+1} = \inf\{i\in\N: \sui i<\frac{x}{\sigma_n} \}$ and $\sigma_{n+1} = \sigma_n\cdot\sui {i_{n+1}}$.
This also means that $\sui {i_{n+1}}$ is the greatest element of $\su$ such that $\sui i<\frac{x}{\sigma_n}$ because $\su$ decreases.

Now, as $\sigma$ monotonously increases and is supposed to converge to $l$, we need to have $ \sigma_{n+1} = \sigma_n\cdot\sui {i_{n+1}}<l$.
More, $\frac{1}{1+\varepsilon} l< \sigma_n$ so that $\frac{1}{1+\varepsilon} l\cdot\sui {i_{n+1}} < l$.
Thus $\sui {i_{n+1}} < 1+\varepsilon$.
In addition, $1+\delta < \frac{x}{\sigma_{n+1}} < \frac{x}{\sigma_{n}}$ because $l>\sigma_{n+1} > \sigma_{n}$.
So $\sui {i_{n+1}} < \frac{x}{\sigma_{n+1}}$.
There cannot exist an $i< i_{n+1}$ such that $\sui i < \frac{x}{\sigma_{n+1}}$ otherwise $\sui i < \frac{x}{\sigma_{n}}$ with $i<i_{n+1}$ which denies the definition of $i_{n+1}$.
Thus, $i_{n+2} = i_{n+1}$ and $\sigma_{n+2}=\sigma_{n+1}\cdot \sui {i_{n+1}} =  \sigma_{n+1}\cdot (\sui {i_{n+1}})^2$.
As $\sigma$ is supposed to converge to $l$, this feature must hold indefinitely.
Only $\sui {i_{n+1}} = 1$ would permit this but this is impossible.
So $l$ cannot be the limit of $\sigma$.
Thus $\sigma$ can only converge to $x$.

In conclusion, $x\in\overline{\Fab}$.
As this is valid for any $x\in [1,a]_\R$, we obtain that $[1,a]_\R \subset \overline{\Fab}$ and, as a consequence, $[1,a]_\R = \overline{\Fab}$.

So $\Fab$ is dense in $[1,a]_\R$.
\end{proof}

Now that we state \reft{thm:F_dense_in_1_a}, we can highlight some of its extensions.

Define $a^k\Fab = \{a^k \cdot f: f\in\Fab\}$ for any positive integer $k$.
An immediate consequence of \reft{thm:F_dense_in_1_a} is that $a^k\Fab$ is dense in $[a^k,a^k+1]_\R$.
So 
\begin{coro}
Given $a$ and $b$ two co-prime positive integers satisfying $1<a<b$,
$$\bigcup\limits_{k\in\N} a^k\Fab = [1,+\infty[_\R.$$
\end{coro}

We can complete this corollary two ways. Either by considering $k\in\Z$ or by considering the sets of the inverse of members of $a^k\Fab$ that is the sets $\{\frac{1}{a^kf}: f\in\Fab\}$. 
Let us choose the first possibility as it is a natural extension of the previous corollary.
Take any $k\in\N$ then $a^{-k}\Fab$ is dense in $[1/a^k, 1/a^{k-1}]$ thus
\begin{coro}
Given $a$ and $b$ two co-prime positive integers satisfying $1<a<b$,
$$\bigcup\limits_{k\in\Z} a^k\Fab = \R^+$$
where $\R^+$ denotes the set of all the positive real numbers.
\end{coro}

\section{Summary and discussion}
Section \ref{sec:definitions} introduces $\phiab$ and $\Fab$ and highlights that 1 is the only integer that $\Fab$ contains.
More important as we use it all along this paper, we state that any member of $\Fab$ has a unique representation of the form $\frac{2^{p+\dab(p)}}{p}$.

Section \ref{sec:phi_isomorphism} defines a binary law $"\star"$ on $\Fab$ that makes of $\phiab : (\N,+,0)\rightarrow(\Fab,\star,1)$ a monoid isomorphism.

Despite the irregularity of $\Fab$, its record holders are bound by a regular expression without particular cases.
Effectively, section \ref{sec:record_holders} introduces the two sequences $\su$ and $\sv$ linked one to each other.
Particularly we state that $\sui i \star \svi i$ is a new record holder.

The following section fix the respective limit of $\su$ and $\sv$.

Last section claims that $\Fab$ is dense in $[1,a]_\R$ inferring by the same way that $\bigcup\limits_{k\in\Z}2^k\Fab$ is dense in $\R^+$.

Note that this paper describes some results which may be used in other fields of interest.
For instance, \refp{prop:sequence_convergence} permits us to find the limits of several sequences.
It seems to be particularly convenient for sequences defined by a product or a fraction as we can see above.

Also, \refl{lem:tu_and_tw_converge_to_1} may be generalized.
Effectively it may be applied to any pair of sequences $u$ and $v$ defined as in \refl{lem:def_tu_and_tw} with the restrictions that $u_0>1$, $v_0>1$ and their is no $k\in\N$ such that there exist $i\in\N$ for which either $u_i = v_i^k$ or $v_i = u_i^k$.
The last condition is necessary so that none of the $u_i$ or $v_i$ equals 1 which would stop the sequences.

More generally on the numbers $a$ and $b$.
Some quick tests tend to prove that the condition $a$ and $b$ are co-prime is not necessary.
If $a = \alpha \gamma$ and $b=\beta \gamma$ with $\alpha < \beta$, $\gamma = gcd(a,b)$ and $\beta$ and $\gamma$ are co-prime, the members of the set $\Fab$ are of the form $\frac{\alpha^{p+\dab(p)}\gamma^{\dab(p)}}{\beta^p}$.
The same properties as for $a$ and $b$ co-prime seem to apply under the condition that $\beta$ and $\gamma$ are co-prime. It is more difficult when this is not the case as the fraction $\frac{\alpha^{p+\dab(p)}\gamma^{\dab(p)}}{\beta^p}$ is not the the smallest representative of its class in $\Q$. 

The same observation seems to hold if $b = \beta a^k$, $\beta$ and $a$ being co-primes. 
Effectively, the members of $\Fab$ would be of the form $\frac{a^q}{\beta^p}$ inferring that $\Fab = \F_{(a,\beta)}$.

But, if it exist $k\in\N$ such that $b=a^k$, for any $p$, we have $\phiab(p) = \frac{a^{pk}}{b^p}=1$ and the set $\Fab=\{1\}$.
So this case is out of interest.
On the same way, if $a=1$, then $\Fab = \{1/b^k: k\in\N\}$ and its behaviour is totally different.

Alas, these result are insufficient to solve our initial problem.
Even, the sequence $\su$ converges to 1 which reduces the possibilities that $2^k/3^p+1/3^p>2^{k-p}/ 3^p + 1$.
But $2^{k-p}/ 3^p + 1$ also decreases. 
Then $2^k/3^p+1/3^p>2^{k-p}/ 3^p + 1$ may still hold for any $p$ where $1<2^k/3^p<2$.

\printbibliography
\end{document}